\documentclass[11pt,a4paper]{amsart}
	
\title{Tube representations and twisting of graded categories}

\date{May 7, 2018}

\author{Jyotishman Bhowmick}
\address{Stat-Math Unit, Indian Statistical Institute, Kolkata, INDIA}
\email{jyotishmanbmath@gmail.com}
\author{Shamindra Ghosh}
\address{Stat-Math Unit, Indian Statistical Institute, Kolkata, INDIA}
\email{shamindra.isi@gmail.com}
\author{Narayan Rakshit}
\address{Stat-Math Unit, Indian Statistical Institute, Kolkata, INDIA}
\email{narayan753@gmail.com}

\author{Makoto Yamashita}
\address{Department of Mathematics, Ochanomizu University, JAPAN}
\email{yamashita.makoto@ocha.ac.jp}

\usepackage{bbm}
\usepackage{mathtools}
\usepackage{amsmath}
\usepackage{amsfonts}
\usepackage{amssymb}
\usepackage{mathrsfs}
\usepackage[normalem]{ulem}
\usepackage{fullpage}
\usepackage{tikz}
\usetikzlibrary{math,cd}

\usepackage{url}

\usepackage[]{enumitem}
\setlist{itemsep=1ex,topsep=1ex} 
\setlist[1]{leftmargin=*}
\setlist[itemize,1]{labelindent=1em}
\setlist[itemize,2]{leftmargin=2pc,labelsep=*}
\setlist[enumerate,1]{label=(\roman*)}
\setlist[enumerate,2]{label=(\alph*)}

\usepackage[nobysame,alphabetic,initials,msc-links]{amsrefs}

\DefineSimpleKey{bib}{how}
\renewcommand{\PrintDOI}[1]{%
  \href{http://dx.doi.org/#1}{{\tt DOI:#1}}%
}
\renewcommand{\eprint}[1]{#1}
\BibSpec{book}{%
    +{}  {\PrintPrimary}                {transition}
    +{.} { \PrintDate}                  {date}
    +{.} { \textit}                     {title}
    +{.} { }                            {part}
    +{:} { \textit}                     {subtitle}
    +{,} { \PrintEdition}               {edition}
    +{}  { \PrintEditorsB}              {editor}
    +{,} { \PrintTranslatorsC}          {translator}
    +{,} { \PrintContributions}         {contribution}
    +{,} { }                            {series}
    +{,} { \voltext}                    {volume}
    +{,} { }                            {publisher}
    +{,} { }                            {organization}
    +{,} { }                            {address}
    +{,} { }                            {status}
    +{,} { \PrintDOI}                   {doi}
    +{,} { \PrintISBNs}                 {isbn}
    +{}  { \parenthesize}               {language}
    +{}  { \PrintTranslation}           {translation}
    +{;} { \PrintReprint}               {reprint}
    +{.} { }                            {note}
    +{.} {}                             {transition}
    +{}  {\SentenceSpace \PrintReviews} {review}
}
\BibSpec{article}{%
    +{}  {\PrintAuthors}                {author}
    +{,} { \textit}                     {title}
    +{.} { }                            {part}
    +{:} { \textit}                     {subtitle}
    +{,} { \PrintContributions}         {contribution}
    +{.} { \PrintPartials}              {partial}
    +{,} { }                            {journal}
    +{}  { \textbf}                     {volume}
    +{}  { \PrintDatePV}                {date}
    +{,} { \issuetext}                  {number}
    +{,} { \eprintpages}                {pages}
    +{,} { }                            {status}
    +{,} { \PrintDOI}                   {doi}
    +{,} { \eprint}        {eprint}
    +{}  { \parenthesize}               {language}
    +{}  { \PrintTranslation}           {translation}
    +{;} { \PrintReprint}               {reprint}
    +{.} { }                            {note}
    +{.} {}                             {transition}
    +{}  {\SentenceSpace \PrintReviews} {review}
}
\BibSpec{collection.article}{%
    +{}  {\PrintAuthors}                {author}
    +{,} { \textit}                     {title}
    +{.} { }                            {part}
    +{:} { \textit}                     {subtitle}
    +{,} { \PrintContributions}         {contribution}
    +{,} { \PrintConference}            {conference}
    +{}  {\PrintBook}                   {book}
    +{,} { }                            {booktitle}
    +{,} { \PrintDateB}                 {date}
    +{,} { pp.~}                        {pages}
    +{,} { }                            {publisher}
    +{,} { }                            {organization}
    +{,} { }                            {address}
    +{,} { }                            {status}
    +{,} { \PrintDOI}                   {doi}
    +{,} { \eprint}        {eprint}
    +{}  { \parenthesize}               {language}
    +{}  { \PrintTranslation}           {translation}
    +{;} { \PrintReprint}               {reprint}
    +{.} { }                            {note}
    +{.} {}                             {transition}
    +{}  {\SentenceSpace \PrintReviews} {review}
}
\BibSpec{misc}{%
  +{}{\PrintAuthors}  {author}
  +{,}{ \textit}      {title}
  +{.}{ }             {how}
  +{}{ \parenthesize} {date}
  +{,} { available at \eprint}        {eprint}
  +{,}{ available at \url}{url}
  +{,}{ }             {note}
  +{.}{}              {transition}
}

\numberwithin{equation}{section}
\numberwithin{figure}{section}
\theoremstyle{plain}
\newtheorem{thm}{Theorem}[section]
\newtheorem{prop}[thm]{Proposition}
\newtheorem{lem}[thm]{Lemma}
\newtheorem{cor}[thm]{Corollary}

\theoremstyle{remark}
\newtheorem{rem}[thm]{Remark}

\theoremstyle{definition}
\newtheorem{defn}[thm]{Definition}
\newtheorem{expl}[thm]{Example}
  
\mathchardef\mhyph="2D
\newcommand{\ra}{\rightarrow}
\newcommand{\rab}{\rangle}

\newcommand{\lab}{\langle}

\newcommand{\mcal}{\mathcal}

\newcommand{\N}{\mathbb N}
\newcommand{\Z}{\mathbb Z}
\newcommand{\C}{\mathbb{C}}
\newcommand{\R}{\mathbb{R}}
\newcommand{\T}{\mathbb{T}}
\newcommand{\cA}{\mathcal{A}}

\newcommand{\cC}{\mathcal{C}}
\newcommand{\cD}{\mathcal{D}}
\newcommand{\cE}{\mathcal{E}}
\newcommand{\cG}{\mathcal{G}}
\newcommand{\cI}{\mathcal{I}}
\newcommand{\cJ}{\mathcal{J}}
\newcommand{\cT}{\mathcal{T}}
\newcommand{\cZ}{\mathcal{Z}}
\newcommand{\bT}{\mathbb{T}}
\newcommand{\un}{\mathbbm{1}}
\newcommand{\scX}{\mathscr{X}}
\newcommand{\scY}{\mathscr{Y}}
\newcommand{\rC}{\mathrm{C}}
\newcommand{\rH}{\mathrm{H}}
\newcommand{\vlon}{\varepsilon}

\newcommand{\vphi}{\varphi}

\newcommand{\ol}[1]{\overline{#1}}
\newcommand{\abs}[1]{\left|{#1}\right |}
\newcommand{\norm}[1]{\left\|{#1}\right\|}
\renewcommand{\t}[1]{\text{#1}}
\newcommand{\flr}[1]{\left\lfloor#1\right\rfloor}

\newcommand{\SU}{\mathrm{SU}}
\newcommand{\TL}{\mathrm{TL}}

\newcommand{\finNat}{\operatorname{Nat}_{00}}
\newcommand{\bbull}{\mathbin{\bar{\bullet}}}
\newcommand{\indcat}[1]{\mathrm{ind}\mhyph{#1}}
\newcommand{\modcat}{\mhyph\mathrm{mod}}
\newcommand{\Id}{\mathrm{Id}} 
\newcommand{\id}{\mathrm{id}} 
\newcommand{\rnC}{\bar{\mathrm{C}}}
\newcommand{\alg}{\mathrm{alg}}
\newcommand{\liesl}{\mathfrak{sl}}
\newcommand{\opos}{\mathrm{op}}
\DeclareMathOperator{\obj}{Obj}
\DeclareMathOperator{\Irr}{Irr}
\DeclareMathOperator{\Ch}{Ch}
\DeclareMathOperator{\Tr}{Tr}
\DeclareMathOperator{\Mor}{Mor}
\DeclareMathOperator{\End}{End}
\DeclareMathOperator{\Ad}{Ad}
\DeclareMathOperator{\dom}{dom}
\DeclareMathOperator{\codom}{codom}
\DeclareMathOperator{\Map}{Map}
\DeclareMathOperator{\sgn}{sgn}
\DeclareMathOperator{\Rep}{Rep}


\begin{document}

\begin{abstract}
We study deformation of tube algebra under twisting of graded monoidal categories.
When a tensor category $\cC$ is graded over a group $\Gamma$, a torus-valued $3$-cocycle on $\Gamma$ can be used to deform the associator of $\cC$.
We show that it induces a $2$-cocycle on the groupoid of the adjoint action of $\Gamma$.
Combined with the natural Fell bundle structure of tube algebra over this groupoid, we show that the tube algebra of the twisted category is a $2$-cocycle twisting of the original one.
\end{abstract}

\maketitle

\section{Introduction}

In the theory of quantum symmetries, the concept of \emph{quantum double} provides a powerful guiding principle to understand various aspects of quantum groups.
There are several ways to precisely realize it mathematically: originally it appeared as the \emph{Drinfeld double} of Hopf algebras which provides a uniform way to produce solutions of the Yang--Baxter equation, and can be regarded as a Hopf algebraic analogue of the complexification of compact semisimple Lie algebras.
A closely related notion is the \emph{Drinfeld center}, as formulated by Drinfeld, Majid, and Joyal and Street independently, in the more general context of tensor categories which generalizes the Drinfeld double through a Tannaka--Krein type duality principle.
Besides producing braided tensor categories, it turned out to have far reaching roles in the theory of tensor categories, such as a natural framework to consider algebra objects encoding various categorical structures of interest.

While there are several other notable approaches, \emph{tube algebra} is perhaps the most concrete (and combinatorial) formalism to define the quantum double for tensor categories, which was introduced by Ocneanu~\cite{MR1317353} in his pioneering study of subfactors and topological quantum field theory.
Although at first sight this looks quite different from other definitions of the Drinfeld center, its precise correspondence was clarified through the subsequent work of Longo and Rehren~\cite{MR1332979}, and Izumi~\cite{MR1782145}, to name a few.

The goal of this work is to understand the change of tube algebra induced by a change of associator on the tensor category.
To be more specific, we consider a change of associator on a tensor category $\cC$ graded by a discrete group $\Gamma$, induced by group 3-cocycles on $\Gamma$, inspired by the work of Kazhdan and Wenzl~\cite{MR1237835} on the classification of the tensor categories with the fusion rules of quantum $\mathrm{SL}(n)$ groups.

A key insight is that the tube algebra $\cT(\cC)$ has a structure of \emph{Fell bundle}~\citelist{\cite{MR1103378}\cite{MR1443836}} over the action groupoid $\cG$ of $\Gamma$ acting on itself by the adjoint action.
Extending the analysis of pointed categories by Bisch, Das, and the second and the third named authors~\cite{MR3699167} (a scheme which was also known to B\'antay and others in a slightly different framework, see~\citelist{\cite{MR1070067}\cite{MR2146291}\cite{MR2443249}}), we show that $\cT(\cC^\omega)$, the tube algebra of $\cC$ with its associator twisted by $\omega$, is isomorphic to the $2$-cocycle twist of $\cT(\cC)$ by a $2$-cocycle on $\cG$ induced by $\omega$.

When $\Gamma$ is a cyclic group, this implies that the Drinfeld center $\cZ(\cC^\omega)$ of the twisted category is equivalent to $\cZ(\cC)$ as a linear category.
The monoidal structure still needs to be modified in the presence of $\omega$, and we give an explicit formula for this twisting.
For one thing, when $\cC$ is braided and $\cZ(\cC)$ admits a good description as a tensor category, this allows us to determine whether $\cC^\omega$ is braided or not which we will explain in detail for the Kazhdan--Wenzl categories (Example~\ref{expl:kazh-wenzl-braided}).

\smallskip
This paper is organized as follows: in Section \ref{sec:prelim} we recall basic concepts and fix notation, which can be easily skipped by an expert.
In Section \ref{sec:ann-alg-nonstr-cat} we present the precise structure of tube algebras for nonstrict tensor categories, which serves as a basis for the twisting argument. Our main result is in Section \ref{sec:twist-ann-alg}. Finally, we present several examples in Section \ref{sec:examples}.

\bigskip
\paragraph{Acknowledgements} It is our pleasure to thank Masaki Izumi, Corey Jones, Madhav Reddy, and Shigeru Yamagami for fruitful discussions at various stages of the project.

\section{Preliminaries}\label{sec:prelim}

In this paper, all categories are assumed to be small, and we mostly work with \emph{C$^*$-categories}, although most of our constructions can be carried out in the setting of semisimple tensor categories over more general coefficients.
We denote the unit circle by $\bT = \{z \in \C \mid \abs{z} = 1 \}$.

\subsection{C\texorpdfstring{$^*$}{*}-tensor categories}

We mainly follow the convention of~\cite{MR3204665}.

When $\cC$ is a category and $X, Y$ are objects of $\cC$, the morphism space from $X$ to $Y$ will be denoted by $\Mor_\cC(X, Y)$ or more simply by $\cC (X,Y)$.
The identity morphism of $X$ is denoted by $1_X$.
For C$^*$-categories, we always assume that direct sum of objects and images of projections exist in the category.
We say that $X \in \obj(\cC)$ is \emph{simple} if $\dim \cC(X, X) = 1$, and that $\cC$ is \emph{semisimple} if $\cC(X, Y)$ is finite dimensional for any $X, Y \in \obj(\cC)$.

A \emph{C$^*$-tensor category} is given by a C$^*$-category $\cC$, a C$^*$-bifunctor $\otimes \colon \cC \times \cC \to \cC$, a distinguished object $\un \in \obj(\cC)$, and natural unitary transformations $\lambda_X \colon \un \otimes X \to X$, $\rho_X\colon X \otimes \un \to X$, and
$$
\alpha_{X,Y,Z} \colon (X \otimes Y) \otimes Z \to X \otimes (Y \otimes Z)
$$
satisfying a standard set of axioms, most notably the \emph{pentagon equation} saying that the diagram
$$
\begin{tikzpicture}[commutative diagrams/every diagram]
\begin{scope}[yscale=0.6,xscale=1,evaluate={\radius=3;}]
\path (0,0) (0,\radius+1); 
  \node (P0) at (0,\radius) {$((U \otimes V)\otimes W) \otimes X $};
  \node (P1) at ({\radius*cos(90+72)-1},{\radius*sin(90+72)}) {$(U \otimes (V\otimes W))\otimes X$};
  \node (P2) at ({\radius*cos(90+2*72)-1.3},{\radius*sin(90+2*72)}) {$U \otimes ((V\otimes W)\otimes X)$};
  \node (P3) at ({\radius*cos(90+3*72)+1.3},{\radius*sin(90+3*72)}) {$U \otimes (V\otimes (W\otimes X))$};
  \node (P4) at ({\radius*cos(90+4*72)+1},{\radius*sin(90+4*72)}) {$(U \otimes V) \otimes (W\otimes X)$};
  \path[commutative diagrams/.cd, every arrow, every label]
    (P0) edge node[swap] {$\alpha_{U,V,W}\otimes 1_X$} (P1)
    (P1) edge node[swap,pos=0.3] {$\alpha_{U,V\otimes W,X}$} (P2)
    (P2) edge node[swap] {$1_U \otimes \alpha_{V,W,X}$} (P3)
    (P4) edge node[pos=0.3] {$\alpha_{U,V,W\otimes X}$} (P3)
    (P0) edge node {$\alpha_{U \otimes V,W,X}$} (P4);
\end{scope}
\end{tikzpicture}
$$
is commutative. In the following we always assume that $\un$ is simple.

When $\cC$ and $\cC'$ are such categories, a C$^*$-tensor functor from $\cC$ to $\cC'$ is given by $(F, F_0, F_2)$ consisting of:
\begin{itemize}
\item  a C$^*$-functor $F\colon \cC \to \cC'$,
\item a unitary morphism $F_0 \colon \un_{\cC'} \to F(\un_{\cC})$, and
\item a natural unitary transformation $F_2 \colon F(X) \otimes F(Y) \to F(X \otimes Y)$,
\end{itemize}
satisfying a standard set of compatibility conditions for $\lambda$, $\rho$, and $\alpha$ of $\cC$ and $\cC'$.
If  $F$ is an equivalence of categories, we say that the above C$^*$-tensor functor is a (unitary monoidal) equivalence of C$^*$-tensor categories.

A variant of Mac Lane's coherence theorem implies that any C$^*$-tensor category is equivalent to a \emph{strict} one, the latter being a C$^*$-tensor category where  $\lambda_X$, $\rho_X$, and $\alpha_{X,Y,Z}$ are all  given by identity morphisms.
In this paper we carry out various constructions for categories with nontrivial associator $\alpha$ but still with $\lambda_X$ and $\rho_X$ given by identity morphisms.
In practice, we still assume that the `starting point' is given by a strict C$^*$-tensor category, and deform it to another one in which $\alpha$ is given by scalar multiples of identity morphisms.

In a \emph{rigid} C$^*$-tensor category $\cC$ any object $X \in \obj(\cC)$ has a dual, that is, there exist $\bar{X} \in \obj(\cC)$ and $R \in \cC(\un, \bar{X} \otimes X)$, $\bar{R} \in \cC(\un, X \otimes \bar{X})$ satisfying the \emph{conjugate equation} for $X$:
\begin{align*}
(1_{\bar{X}} \otimes \bar{R}^*) \alpha_{\bar{X},X,\bar{X}} (R \otimes 1_{\bar{X}}) &= 1_{\bar{X}},&
(1_X \otimes R^*) \alpha_{X,\bar{X},X} (\bar{R} \otimes 1_X) &= 1_X.
\end{align*}
(Recall that we are assuming $\lambda_X = 1_X = \rho_X$ in $\cC$.)
The number $d(X) = \min_{(R, \bar{R})} \norm{R} \norm{\bar{R}}$ is called the \emph{intrinsic dimension} of $X$, where the minimum is taken over all pairs $(R, \bar{R})$ as above.
A solution $(R, \bar{R})$ satisfying $\norm{R} = d(X)^{1/2} = \norm{\bar{R}}$ is called a \emph{standard} solution, which we denote by $(R_X, \bar{R}_X)$.
Standard solutions are unique up to unitary morphisms.
It implies that for any $X \in \obj(\cC)$, the functional $\Tr_X$ on $\cC(X)$ defined by
$$
R^*_X \circ (1_{\bar{X}} \otimes x) \circ R_X = \Tr_X(x) 1_{\mathbbm 1} \quad (x \in \cC(X))
$$
is a tracial positive functional, called the \emph{canonical picture trace} or \emph{categorical trace}.
Having fixed $(R_X, \bar{R}_X)$ and $(R_Y, \bar{R}_Y)$, whenever $x \in \cC(X, Y)$ we denote by $x^\vee$ the morphism in $\cC(\bar{Y}, \bar{X})$ satisfying $(1_{\bar{X}} \otimes x) R_X = (x^\vee \otimes 1_Y) R_Y$.

Rigid C$^*$-tensor categories with simple units are automatically semisimple, and we work within this framework.

\subsection{Graded tensor categories}

In a C$^*$-category $\cC$, two full subcategories $\cD$ and $\cE$ are called \textit{mutually orthogonal} if $\cC (D,E) = \{0\}$ holds for all $D \in \obj (\cD)$ and $E \in \obj (\cE)$.
We write $\cC = \cD \oplus \cE$ if in addition any object in $\cC$ is isomorphic to a direct sum $D \oplus E$ for some $D \in \obj(\cD)$ and $E \in \obj(\cE)$, and say that $\cC$ is a direct sum of $\cD$ and $\cE$.
Of course, this has a straightforward extension to a family of subcategories $\{\cC_j \}_{j\in J}$.

\begin{defn}[cf.~\cite{MR2674592}]\label{grgrad}
Let $\Gamma$ be a group.
A C$^*$-tensor category $\cC$ is called \textit{$\Gamma$-graded} if there exists a collection of mutually orthogonal full subcategories $\{\cC_\gamma\}_{\gamma\in \Gamma}$ such that
\begin{enumerate}
\item $\cC = \bigoplus_{\gamma \in \Gamma} \cC_\gamma$, and
\item $X \otimes Y$ is isomorphic to an object in $\cC_{\gamma \eta}$ for all $\gamma, \eta \in \Gamma$, $X \in \obj(\cC_\gamma)$, and $Y \in \obj(\cC_\eta)$.
\end{enumerate}
\end{defn}

\begin{rem}\label{trivdual}
The tensor unit $\un$, being simple, must belong to $\obj(\cC_e)$ by conditions (i) and (ii) where $e$ is the unit of $\Gamma$.
It then follows that $\bar{X} \in \obj (\cC_{\gamma^{-1}})$ (if it exists) for all $X \in \obj (\cC_\gamma)$.
\end{rem}

Note that if $\cC$ is graded over $\Gamma$, the `support' of a grading defined as
$$
\{\gamma \in \Gamma \mid \cC_\gamma \text{ contains a nonzero object}\}
$$
is a subgroup of $\Gamma$.
Without losing generality we may assume that the support is always the entire group.
Let us also note that when $\cC$ is $\Gamma$-graded and $\pi\colon \Gamma \to \Lambda$ is a (surjective) group homomorphism, we obtain a $\Lambda$-grading on $\cC$ by setting $\cC_{\gamma'} = \bigoplus_{\pi(\gamma) = \gamma'} \cC_\gamma$ for $\gamma' \in \Lambda$.

Any rigid semisimple C$^*$-tensor category $\cC$ with simple unit admits a grading $\cC = \bigoplus_{\gamma \in \Ch(\cC)} \cC_\gamma$ over the \emph{universal grading group} $\Ch(\cC)$ (also called the \emph{chain group}) in the sense that any $\Gamma$-grading on $\cC$ is induced by a unique homomorphism $q \colon \Ch(\cC) \to \Gamma$ as above.

Concretely, $\Ch(\cC)$ is constructed as follows.
Fix a set $\cI = \Irr(\cC)$ of representatives from isomorphism classes of simple objects.
The elements of $\Ch(\cC)$ are represented by symbols $[X]$ for $X \in \cI$, different symbols possibly representing the same element, subject to the rule $[Z] = [X] [Y]$ whenever there is a nonzero morphism from $Z$ to $X\otimes Y$, or equivalently, when $Z$ is isomorphic to a subobject of $X \otimes Y$.
From this description it immediately follows that $[\un]$ is the unit of $\Ch(\cC)$, $[X]^{-1} = \left[ \bar{X} \right]$.
For $\gamma \in \Ch(\cC)$, define $\cC_\gamma$ as the full subcategory of $\cC$, whose objects decompose as direct sum of simple objects $X$ satisfying $[X] = \gamma$.
Clearly, $\{\cC_\gamma \}_{\gamma \in \Ch(\cC)}$ is a mutually orthogonal collection of subcategories satisfying conditions of Definition \ref{grgrad}.

\begin{expl}[Temperley--Lieb category]
\label{expl:TL-and-SUq2}
Consider the \emph{Temperley--Lieb category} $\cC^{\TL}_{\delta, \vlon}$ with modulus $\delta \geq 2$ and the associativity sign $\vlon \in \{\pm 1\}$, which can also be seen as the category of finite dimensional representations of the compact quantum group $\SU_q(2)$~\cite{MR890482}, with $q \in \R^\times$ satisfying $\abs{q + q^{-1}} = \delta$ and $\vlon = -\sgn(q)$.
The set of isomorphism classes of simple objects of $\cC^{\TL}_{\delta, \vlon}$ is countably infinite.
In fact, there is a set $\{X_n\}_{n\geq 0}$ of representatives of such classes indexed by nonnegative integers satisfying:
\begin{enumerate}
\item $X_0$ is the trivial object $\un$,
\item $X_m \otimes X_n  \cong X_{\abs {m-n} } \oplus X_{\abs {m-n} +2 } \oplus \cdots \oplus X_{(m+n) -2} \oplus X_{m+n}$, and
\item $\ol{X_m} \cong X_m$
\end{enumerate}
for all $m, n \geq 0$.
The parameters $\delta$ and $\vlon$ (which do not affect $\Ch(\cC^{\TL}_{\delta, \vlon})$) appear as follows: there is a morphism $R \colon X_0 \to X_1 \otimes X_1$, unique up to $\bT$, such that
\begin{align*}
R^* R &= \delta 1_{X_0},& (1_{X_1} \otimes R^*) \alpha_{X_1, X_1, X_1} (R \otimes 1_{X_1}) &= \vlon 1_{X_1}
\end{align*}
Condition (ii) implies that every $X_n$ is a subobject of $X_1^{\otimes n}$.
It follows that the universal grading group $\Ch(\cC^{\TL}_{\delta, \vlon})$ has to be cyclic generated by $[X_1]$.
Again, using condition (ii), one can easily show that $[X_n]$ is either equal to $[\un]$ or to $[X_1]$ according as $n$ is even or odd, which implies $\Ch(\cC^{\TL}_{\delta, \vlon}) \cong \Z/2\Z$.
\end{expl}

\subsection{Group and groupoid cohomology}

Let $\Gamma$ be a (discrete) group, and $M$ be a left $\Gamma$-module.
The group cochain complex $\rC^*(\Gamma; M)$ is given by the mapping spaces $\rC^n(\Gamma; M) = \Map(\Gamma^n, M)$ endowed with the coboundary maps $\delta^n \colon \rC^n(\Gamma; M) \to \rC^{n+1}(\Gamma; M)$ given by
\begin{multline}
\label{group-cocycle-cbdry}
\delta^n ( \phi ) ( \gamma_1, \gamma_2, \ldots, \gamma_{n + 1} ) = \gamma_1 \phi ( \gamma_2, \ldots, \gamma_{n + 1} ) \\
+ \sum^n_{i = 1} ( - 1 )^i \phi ( \gamma_1, \gamma_2, \ldots, \gamma_{i - 1}, \gamma_i \gamma_{i + 1}, \gamma_{i + 2}, \ldots, \gamma_{n + 1} ) + ( - 1 )^{n + 1} \phi ( \gamma_1, \ldots, \gamma_n  ).
\end{multline}
A cochain $\phi \in \rC^n(\Gamma; M)$ is  \emph{normalized} if $\phi(\gamma_1, \ldots, \gamma_n) = 0$ whenever one of $\gamma_i$ is $e$.
The subcomplex $\bar{\rC}^*(\Gamma; M)$ of normalized cochains have the same cohomology as $\rC^*(\Gamma; M)$.

Next let $\cG$ be a groupoid.
As usual let us denote by $\cG^{(0)}$ its object set, and by $\cG^{(n)}$ the set of $n$-tuples $(g_1, \ldots, g_n)$ of composable arrows.
Thus, there are maps $\dom$ and $\codom$ from $\cG = \cG^{(1)}$ to $\cG^{(0)}$ so that $g_1$ and $g_2$ are composable if and only if $\dom(g_1) = \codom(g_2)$.
We also identify $\cG^{(0)}$ as a subset of $\cG$ by the embedding $x \mapsto \id_x$.

When $M$ is a commutative group, the complex of normalized cochains on $\cG$ with coefficient $M$ is given by~\cite{MR0255771}
$$
\rnC^n(\cG; M) = \{ \psi\colon \cG^{(n)} \to M \mid \psi(g_1, \ldots, g_n) = 0 \text{ if } \exists i \colon g_i \in \cG^{(0)} \},
$$
together with the differential $\delta^n\colon \rnC^n(\cG; M) \to \rnC^{n+1}(\cG; M)$ given by
\begin{multline*}
\delta^n(\psi)(g_1, \ldots, g_{n+1}) = \psi(g_2, \ldots, g_{n+1}) + \sum_{i=1}^n (-1)^i \psi(g_1, \ldots, g_{i} g_{i+1}, \ldots, g_{n+1})\\
+ (-1)^{n+1} \psi(g_1, \ldots, g_n).
\end{multline*}
Its cohomology is denoted by $\rH^*(\cG; M)$.
Of course, we could have used the complex of non-normalized cochains, $\rC^n(\cG; M) = \Map(\cG^{(n)}, M)$, to compute $\rH^*(\cG; M)$.

We are particularly interested in the \emph{action groupoid} $\Gamma \ltimes X$ arising from an action of a group $\Gamma$ on a set $X$.
This groupoid has the object set $X$, and $\Gamma \ltimes X = \Gamma \times X$ as a set, with domain and codomain maps defined by $\dom(\gamma, x) = x$ and $\codom(\gamma, x) = \gamma x$, and composition given by $(\gamma, x) . (\eta, y) = (\gamma \eta, y)$ whenever $x = \eta y$.
In this case $\rnC^n(\cG; M)$ is nothing but the space $\rnC_\Gamma^n(X; M)$ of normalized equivariant cochains on the (discrete) set $X$ with coefficient in $M$, or what amounts to the same thing, the space $\rnC^n(\Gamma; \Map(X, M))$ of normalized cochains on $\Gamma$ with coefficient $\Map(X, M)$, where $\Gamma$ acts on $\Map(X, M)$ from left by $(\gamma f)(x) = f(\gamma^{-1} x)$.
An explicit cochain isomorphism $ \bar {\rC}^n (\cG; M) \ni \psi \mapsto \widetilde{\psi} \in \bar{\rC}^n(\Gamma; \Map(X,M)) $ is given by
\begin{equation}\label{group-groupoid-link}
\widetilde{\psi}[\gamma_1, \ldots , \gamma_n](x_0) = \psi (g_1, \ldots , g_n),
\end{equation}
where $g_i = (\gamma_i , x_i) \in \cG$ and $x_i = \gamma_i^{-1} x_{i-1}$ for all $1 \leq i \leq n$.

\subsection{Twisting monoidal categories by group \texorpdfstring{$3$}{3}-cocycles}

Again let $\Gamma$ be a group, and let $\omega$ be a cocycle in $\rnC^3(\Gamma; \bT)$.
So $\omega$ is a map from $\Gamma \times \Gamma \times \Gamma$ to $\bT$ satisfying
\[
\omega(\gamma_1,\gamma_2,\gamma_3) \omega(\gamma_1, \gamma_2 \gamma_3, \gamma_4)\omega(\gamma_2, \gamma_3, \gamma_4) = \omega(\gamma_1 \gamma_2, \gamma_3, \gamma_4) \omega(\gamma_1, \gamma_2, \gamma_3 \gamma_4) \quad (\gamma_i \in \Gamma),
\]
and $\omega(\gamma_1, \gamma_2, \gamma_3) = 1$ whenever at least one of $\gamma_i$ is the unit $e$ of $\Gamma$.

When $\cC$ is a $\Gamma$-graded C$^*$-tensor category, we can consider a new category $\cC^\omega$ which
\begin{itemize}
\item is same as $\cC$ as a C$^*$-category,
\item has the same tensor bifunctor and left and right unit constraints, but
\item has the associativity constraint $\alpha^\omega$ given as a twist of that of $\cC$ by the cocycle $\omega$:
\[
\alpha^\omega_{X_1,X_2,X_3} = \omega (\gamma_1,\gamma_2,\gamma_3)  \alpha_{X_1,X_2,X_3} \quad (X_i \in \obj (\cC_{\gamma_i})).
\]
\end{itemize}
The new associator $\alpha^\omega$ still satisfies the same compatibility conditions as $\alpha$ because $\omega$ is a normalized $3$-cocycle.

Note that $\cC^\omega$ has the same fusion rules as $\cC$, and in particular it inherits the $\Gamma$-grading of $\cC$.
However, typically they are monoidally equivalent (if and) only if $\omega$ is a coboundary.

\begin{prop}
\label{prop:tw-cat-rigid}
When $\cC$ is a rigid C$^*$-tensor category, $\cC^\omega$ is also rigid.
\end{prop}

\begin{proof}
It is enough to see that for any $\gamma \in \Gamma$ and $X \in \cC_\gamma$, there is a dual object of $X$ in $\cC^\omega$.
By our standing assumption $X$ admits a dual object in $\cC$, say given by $(Y, R, \bar{R})$.
Then $Y$ is an object of $\cC_{\gamma^{-1}}$.
Moreover we have $\omega(\gamma, \gamma^{-1}, \gamma) = \bar{\omega}(\gamma^{-1}, \gamma, \gamma^{-1})$, since by the $3$-cocycle identity on $(\gamma, \gamma^{-1}, \gamma, \gamma^{-1})$ we have
$$
\omega(\gamma^{-1}, \gamma, \gamma^{-1}) \bar{\omega}(e, \gamma, \gamma^{-1}) \omega(\gamma, e, \gamma^{-1}) \bar{\omega}(\gamma, \gamma^{-1}, e) \omega(\gamma, \gamma^{-1}, \gamma) = 1
$$
and the terms involving $e$ are $1$ by the normalization condition on $\omega$.
Consequently $R' = R$ and $\bar{R}' = \omega(\gamma^{-1}, \gamma, \gamma^{-1}) \bar{R}$ satisfy
\begin{align*}
(1_{Y} \otimes \bar{R}^{\prime *}) \alpha^\omega_{Y, X, Y} (R' \otimes 1_{Y}) &= \bar{\omega}(\gamma^{-1}, \gamma, \gamma^{-1}) \omega(\gamma^{-1}, \gamma, \gamma^{-1}) 1_{Y} = 1_{Y},\\
(1_{X} \otimes R^{\prime *}) \alpha^\omega_{X, Y, X} (\bar{R}' \otimes 1_{X}) &= \omega(\gamma, \gamma^{-1}, \gamma) \omega(\gamma^{-1}, \gamma, \gamma^{-1}) 1_{X} = 1_X.
\end{align*}
Thus, the triple $(Y, R', \bar{R}')$ gives a dual of $X$ in $\cC^\omega$.
\end{proof}

\begin{expl}
The first basic example comes from pointed categories, that is, when the tensor product of every simple object with its dual is isomorphic to the unit object.
The category $\cC_\Gamma$ of finite dimensional $\Gamma$-graded Hilbert spaces is such an example.
Generally, the isomorphism classes of simple objects in a pointed category automatically becomes a group with respect to tensor product where the inverse is given by taking dual; this is the universal grading group.
In fact, any pointed C$^*$-tensor category is unitarily monoidally equivalent to $\cC_\Gamma^\omega$ for some group $\Gamma$ and a (normalized) $\bT$-valued $3$-cocycle $\omega$ on $\Gamma$.
\end{expl}

\section{Some formulas for nonstrict categories}
\label{sec:ann-alg-nonstr-cat}

\subsection{Annular algebras}

In~\cite{MR3447719}, annular algebras were defined assuming strictness of the tensor structure.
While the C$^*$-variant of Mac Lane's coherence theorem guarantees that there is no loss of generality in doing so, for our purposes it will be useful to work out concrete formulas for non-strict ones.
In this section $\cC$ denotes a rigid C$^*$-tensor category with simple unit, and we denote a set of representative of isomorphism classes of simple objects in $\cC$ by $\cI = \Irr(\cC)$.

Suppose that $\scX = \{X_j\}_{j \in \cJ}$ is a family of objects of $\cC$ which is \emph{full} in the sense that any $S \in \cI$ is isomorphic to a subobject of $X_j$ for some $j$.
In order to simplify the notation, let us use the index $j$ in place of $X_j$ when they appear in subscripts, so that $\alpha_{S, T, X_j}$ becomes $\alpha_{S, T, j}$ for example.

For $j , k \in \cJ$, let $\cA_{k ,j}$ denote the quotient of the vector space
\[
\bigoplus_{S \in \obj(\cC)} \cC(S \otimes X_{j}, X_{k} \otimes S)
\]
over the subspace generated by elements of the form
\begin{equation}
\label{eq:rel-in-ann-alg}
f (f' \otimes 1_{j}) - (1_{k} \otimes f')  f, \quad \left(f \in \cC (T \otimes X_{j}, X_{k} \otimes S), f' \in\cC (S,T)\right).
\end{equation}
Further, we will write $\cA^S_{k,j}$ for the image of $\cC (S \otimes X_{j}, X_{k} \otimes S)$ under the quotient map, and the map $\psi^S_{k,j}\colon \cC (S \otimes X_{j}, X_{k} \otimes S) \ra \cA^S_{k,j}$ will stand for the restriction of the quotient map.

The vector space $\cA_{k,j}$ can be identified with the direct sum $\bigoplus_{S \in \cI} \cC(S \otimes X_{j}, X_{k} \otimes S)$,  hence each $\psi^S_{k,j}$ is injective (thus bijective) for $S \in \cI$.
Indeed, for a general $S \in \obj(\cC)$ we can take an irreducible decomposition $S \cong \bigoplus_\alpha S_\alpha$ with $S_\alpha \in \cI$.
Let $v_\alpha \colon S_\alpha \to S$ be the corresponding isometry.
When $f \in \cC(S \otimes X_{j}, X_{k} \otimes S)$, we have $f = \sum_\alpha f (v_\alpha v_\alpha^* \otimes 1_j)$.
From the formula \eqref{eq:rel-in-ann-alg} with $f (v_\alpha \otimes 1_j)$ in place of $f$ and $v_\alpha^* \otimes 1_j$ in place of $f'$, we see that $f$ and $\sum_\alpha (1_k \otimes v_\alpha^*) f (v_\alpha \otimes 1_j)$ represent the same element in $\cA_{k,j}$.
This already shows that $\cA_{k,j}$ is a quotient of $\bigoplus_{S \in \cI} \cC(S \otimes X_{j}, X_{k} \otimes S)$.
Since for different $S, T \in \cI$ there is no nonzero morphism $f' \in \cC(S, T)$, the images of $\cC(S \otimes X_{j}, X_{k} \otimes S)$ for $S \in \cI$ are linearly independent.

Let us put $\cA = \cA(\scX) = \bigoplus_{j,k \in \cJ} \cA_{k,j}$, and define a bilinear operation $\bullet$ on $\cA$ characterized by
\begin{equation}\label{annmult}
\psi^S_{m,k} (f)  \mathbin{\bullet} \psi^T_{k',j} (f') = \delta_{k,k'} \psi^{S\otimes T}_{m,j} \left( \alpha_{m, S,T} (f \otimes 1_T) \alpha_{S, k, T}^{-1}  (1_S \otimes f') \alpha_{S, T, j} \right)
\end{equation}
for $f \in \cC (S \otimes X_{k}, X_{m} \otimes S)$ and $f' \in \cC (T \otimes X_{j}, X_{k'} \otimes T)$.
This is a straightforward adaptation of the product formula in~\cite{MR3447719}, by putting the associator $\alpha$ in appropriate places so that the overall formula makes sense in the current nonstrict setting.
Independence of the definition of $\bullet$ on the choice of $f$ and $f'$ easily follows from the naturality of $\alpha$.

\smallskip
The space $\cA_{k, j}$ can be presented in a more categorical way as follows.
Let $\iota \otimes X_{j}$ be the endofunctor of $\cC$ given by $Y \mapsto Y \otimes X_{j}$, and similarly consider $X_{k} \otimes \iota$.
Denote by $\finNat(\iota \otimes X_{j}, X_{k} \otimes \iota)$ the space of natural transformations $(f_Y \colon Y \otimes X_{j} \to X_{k} \otimes Y)_{Y \in \obj(\cC)}$ from the functor from $\iota \otimes X_{j}$ to $X_{k} \otimes \iota$, such that $f_X = 0$ for $X \in \cI$ except for finitely many of them.
Then, using semisimplicity of $\cC$, one can check that the map
$$
\finNat(\iota \otimes X_{j}, X_{k} \otimes \iota) \to \cA_{k, j},\quad
f \mapsto \sum_{S \in \cI} \psi^S_{k,j}(f_S)
$$
is a vector space isomorphism.
Moreover, $\bigoplus_{j,k \in \cJ} \finNat(\iota \otimes X_{j}, X_{k} \otimes \iota)$ can be equipped with a canonical algebra structure  $\bullet$ for which the above mentioned  map becomes an algebra isomorphism.
Concretely, for $f \in \finNat(\iota \otimes X_{k}, X_{m} \otimes \iota)$ and $f' \in \finNat(\iota \otimes X_{j}, X_{k} \otimes \iota)$, $f \mathbin{\bullet} f'$ is characterized by the following: given $X \in \cI$, the morphism $(f \bullet f')_X$ can be expressed as
$$
(f \mathbin{\bullet} f')_X = \sum_{\mathclap{\substack{Y,Z \in \cI,\\v_a\colon X \to Y \otimes Z}}} (1_{m} \otimes v_a^*) \alpha_{m,Y,Z} (f_Y \otimes 1_Z) \alpha^{-1}_{Y,k,Z} (1_Y \otimes g_Z) \alpha_{Y,Z,j} (v_a \otimes 1_{j}),
$$
where $v_a \colon X \to Y \otimes Z$ runs through an orthonormal basis of $\cC(X, Y \otimes Z)$ for each $Y$ and $Z$.
It follows that, for $f \in \cC (S \otimes X_{k}, X_{m} \otimes S)$ and $f' \in \cC (T \otimes X_{j}, X_{k} \otimes T)$, their product in terms of $\psi$ can be written as
\begin{equation*}
\psi^S_{m,k} (f) \mathbin{\bullet} \psi^T_{k,j} (f') = \sum_{U, w} \psi^{U}_{m,j} \left( \left(1_{m} \otimes w^*\right) \alpha_{m, S, T} (f \otimes 1_T)  \alpha_{S,k,T}^{-1} (1_S \otimes f') \alpha_{S, T, j} \left(w \otimes 1_{j}\right) \right),
\end{equation*}
where $U$ runs through $\cI$ and $w$ through an orthonormal basis of $\cC (U, S \otimes T)$.

\begin{lem}\label{ass}
$(\cA , \bullet)$ is an associative algebra.
\end{lem}

\begin{proof}
Take indices $j, k, m, n \in \cJ$, objects $S, T, U \in \obj(\cC)$, and morphisms
\begin{align*}
f &\in \cC(S \otimes X_{m}, X_{n} \otimes S), &
f' &\in \cC(T \otimes X_{k}, X_{m} \otimes T), &
f'' &\in \cC(U \otimes X_{j}, X_{k} \otimes U).
\end{align*}
On the one hand, $(\psi^S_{n, m}(f) \bullet \psi^T_{m, k}(f')) \bullet \psi^U_{k, j}(f'')$ is the image under $\psi^{(S \otimes T) \otimes U}_{n, j}$ of
\begin{equation}
\label{eq:bull-assoc-1}
\alpha_{n, S \otimes T, U}  \left((\alpha_{n, S,T} (f \otimes 1_T) \alpha_{S, m, T}^{-1}  (1_S \otimes f') \alpha_{S, T, k}) \otimes 1_U\right) \alpha_{S\otimes T, k, U} (1_{S \otimes T} \otimes f'') \alpha_{S \otimes T, U, j}.
\end{equation}

On the other, $\psi^S_{n, m}(f) \bullet (\psi^T_{m, k}(f') \bullet \psi^U_{k, j}(f''))$ is the image under $\psi^{S \otimes (T \otimes U)}_{n, j}$ of
\begin{equation}
\label{eq:bull-assoc-2}
\alpha_{n, S, T \otimes U} (f \otimes 1_{T \otimes U}) \alpha_{S,m, T \otimes U}^{-1} \left(1_S \otimes (\alpha_{m, T,U} (f' \otimes 1_U) \alpha_{T, k, U}^{-1}  (1_T \otimes f'') \alpha_{T, U, j}) \right) \alpha_{S, T \otimes U,j}.
\end{equation}
Using $\psi^{S \otimes (T \otimes U)}_{n,j}((1_{n} \otimes \alpha) f' (\alpha^{-1} \otimes 1_{j})) = \psi^{(S \otimes T) \otimes U}_{n,j}(f')$ together with the naturality and the pentagon equation for $\alpha$, we indeed obtain that \eqref{eq:bull-assoc-1} and \eqref{eq:bull-assoc-2} are equal.

Let us indicate the first step.
In the expression \eqref{eq:bull-assoc-1}, by the naturality of $\alpha$ we can insert $\alpha_{S, m \otimes T, U}^{-1}$ and $\alpha_{S, T \otimes k, U}$ around $(1_S \otimes f') \otimes 1_U$, and $\alpha_{S, T, k \otimes U}^{-1}$ and $\alpha_{S, T, U \otimes j}$ around $1_{S \otimes T} \otimes f''$.
Then between $(1_S \otimes (f' \otimes 1_U))$ and $(1_S \otimes (1_T \otimes f''))$ we have
$$
\alpha_{S, T \otimes k, U} (\alpha_{S, T, k} \otimes 1_U) \alpha_{S \otimes T, k, U}^{-1} \alpha_{S, T, k \otimes U}^{-1} \colon S \otimes (T \otimes (X_{k} \otimes U)) \to S \otimes ((T \otimes X_{k}) \otimes U),
$$
which is equal to $1_S \otimes \alpha_{T, k, U}^{-1}$ by the pentagon identity.
This is exactly what we have between in $f'$ and $f''$ in \eqref{eq:bull-assoc-2}.
The rest of the proof proceeds in a similar way.
\end{proof}

Let us next describe the $*$-structure on $\cA$.
We will denote this one by $\#$ in order to avoid confusion with the involution of morphisms of $\cC$.
Define a conjugate linear map $\# \colon \cA \ra \cA$ by sending $\psi^{S}_{k,j} (f) \in \cA^{S}_{k,j}$ to
\begin{equation}
\label{eq:invol-on-ann-alg}
\psi^{\bar{S}}_{j,k} \biggl( \Bigl( \bigl( (R_S^* \otimes 1_{j}) \alpha_{\bar{S}, S, j}^{-1} \bigr) \otimes 1_{\bar{S}} \Bigr) \bigl( ( 1_{\bar{S}} \otimes f^* ) \otimes 1_{\bar{S}} \bigr) \alpha_{\bar{S}, k \otimes S , \bar{S}}^{-1} \Bigl( 1_{\bar{S}} \otimes \bigl(\alpha_{k,S,\bar{S}}^{-1} ( 1_{k} \otimes \bar{R}_S) \bigr) \Bigr) \biggr)
\end{equation}
in $\cA^{\bar S}_{j,k}$, where $(R_S , \bar{R}_S)$ is a standard solution to conjugate equations for $S$.
Note that by the naturality of $\alpha$ and the quotient map $\psi$, this becomes well-defined as well as independent of the choice of $\bar{S}$ and $(R_S, \bar{R}_S)$.
Moreover, $\#$ is indeed an antimultiplicative involution; to see this, one argues along the same lines as in the proof of Lemma \ref{ass}, using the fact that $\{w_i^{\vee *}\}_i$ is an orthonormal basis of $\cC(\bar{U}, \bar{T} \otimes \bar{S})$ if $\{w_i\}_i$ is a one of $\cC(U, S \otimes T)$ with respect to a standard solution $(R_U, \bar{R}_U)$ for $U$, $(R_T, \bar{R}_T)$ for $T$, and the associated one
\begin{equation}
\label{eq:std-sol-for-prod-with-assoc}
\left(\alpha_{\bar T \otimes \bar S, S, T} (\alpha_{\bar T, \bar S, S}^{-1} \otimes 1_T) \bigl((1_{\bar T} \otimes R_S) \otimes 1_T \bigr) R_T, \alpha_{S \otimes T, \bar T, \bar S} (\alpha_{S, T, \bar T}^{-1} \otimes 1_{\bar S}) \bigl((1_S \otimes \bar{R}_T) \otimes 1_{\bar S} \bigr) \bar{R}_S \right)
\end{equation}
for $S \otimes T$.

\begin{prop}
\label{prop:ann-alg-indep-tensor-equiv}
Let $F \colon \cC \to \cC'$ be a unitary monoidal equivalence of rigid C$^*$-tensor categories with simple units.
Furthermore let $\scX = \{X_j\}_{j \in \cJ}$ be a full family in $\cC$, and put $\scX' = \{F(X_j)\}_{j \in \cJ}$.
Then $\cA(\scX)$ is isomorphic to $\cA(\scX')$.
\end{prop}

\begin{proof}
By definition $F$ comes with a C$^*$-functor $\cC \to \cC'$ (again denoted by the same symbol $F$) and a natural unitary transformation $F_2 \colon F(X) \otimes F(Y) \to F(X \otimes Y)$ satisfying compatibility conditions for the monoidal structures.
We then have a linear map
$$
\cC(S \otimes X_j, X_k \otimes S) \to \cC'(F(S) \otimes F(X_j), F(X_k) \otimes F(S)), \quad f \mapsto F_2^{-1} F(f) F_2
$$
which induces a linear map $\cA(\scX)_{k, j} \to \cA(\scX')_{k, j}$.
A tedious but straightforward computation yields that this is indeed a $*$-isomorphism.
\end{proof}

\begin{defn}[cf.~\citelist{\cite{MR1317353}\cite{MR1929335}\cite{MR3447719}}]
The $*$-algebra $(\cA(\scX), \bullet, \#)$ is called the \textit{annular algebra of $\cC$ associated with the family $\scX = \{X_j\}_{j \in \cJ}$}.
When $\scX = \Irr \cC$, a set of representatives of the simple objects, we call $\cT = \cT(\cC) = \cA(\scX)$ the \emph{tube algebra} of $\cC$.
\end{defn}

The results of~\cite{MR3447719} apply to our construction by Proposition~\ref{prop:ann-alg-indep-tensor-equiv}.
Let us summarize some important consequences:
\begin{itemize}
\item any element of $\cA(\scX)$ has a uniform bound on the norm under the $*$-representations of $\cA$; hence $\cA$ admits a universal C$^*$-envelope.
\item there is a canonical faithful positive trace on $\tau \colon \cA \to \C$ given by
\[
\cA^{S}_{k,j} \ni \psi^{S}_{k,j} (f) \mapsto \delta_{j, k} \sum_{w} \Tr_{j} \bigl( \left(1_{j} \otimes w^* \right) f \left(w \otimes 1_{j} \right)\bigr),
\]
where $w$ runs through an orthonormal basis of $\cC (\un, S)$.
(We suppressed the structure morphisms for $\un$ as before.)
Note that simplicity of $\un$ is crucially used here, and the definition of $\tau$ is indeed independent of the choice of the orthonormal basis.
\item the subspaces $\cA^S_{k, j}$ of $\cA_{k, j}$ for different $S \in \cI$ are orthogonal with respect to the inner product defined by $\lab a_1 , a_2 \rab = \tau \bigl(a_1 \bullet a^\#_2\bigr)$.
\item if $\scX = \{X_j\}_{j \in \cJ}$ and $\scY = \{Y_{j'}\}_{j' \in \cJ'}$ are two full families, the algebras $\cA(\scX)$ and $\cA(\scY)$ are strongly Morita equivalent.
\end{itemize}

\subsection{Half-braiding and monad}
\label{sec:ind-half-br-monad}

When $X$ is an object of a C$^*$-tensor category $\cC$, a unitary half-braiding on $X$ is a natural family of unitary morphisms $c_Y \colon Y \otimes X \to X \otimes Y$ for $Y \in \obj(\cC)$ such that
\begin{equation}
\label{eq:half-br-mult-formula}
\alpha_{X, Y, Z} (c_Y \otimes 1_Z) \alpha_{Y, X, Z}^{-1} (1_Y \otimes c_Z) \alpha_{Y, Z, X} = c_{Y \otimes Z}.
\end{equation}
A morphism of half-braidings from $(X, c)$ to $(X', c')$ is a morphism $x$ from $X$ to $X'$ (in $\cC$) which satisfies $(x \otimes 1_Y) c_Y = c'_Y (1_Y \otimes x)$ for all $Y$.
The category $\cZ(\cC)$ of the pairs $(X, c)$ as above is the (unitary) \emph{Drinfeld center}.

Recall that we can enlarge $\cC$ to the ind-category $\indcat{\cC}$ whose objects are \emph{direct limits} $\varinjlim X_i$ for some inductive system $(X_i, v_{ji})_{i,j \in \Lambda}$ of objects labeled over a directed set $\Lambda$, with connecting \emph{isometries} $v_{ji} \colon X_i \to X_j$ for $i < j$~\cite{MR3509018}.
For semisimple $\cC$, with a representative of irreducible classes $\cI$ as before, such a limit can always be represented as `infinite direct sum' $\bigoplus_{S \in \cI} H_S \otimes S$.
Here $H_S \otimes S$ is an amplification of $S$ by a Hilbert spaces $H_S$, and the morphism space of two such direct sums is concretely given by
$$
\Mor_{\indcat{\cC}} \Bigl( \bigoplus_{S \in \cI} H_S \otimes S, \bigoplus_{S \in \cI} H'_S \otimes S \Bigr) = \ell_\infty\mhyph\prod_{S \in \cI} B(H_S, H'_S),
$$
where the right hand side denotes the space of uniformly bounded sequences of operators with respect the operator norm.

Then the category $\indcat{\cC}$ is again a (non-rigid) C$^*$-tensor category, and we can again form the Drinfeld center category $\cZ(\indcat{\cC})$, which is often the `correct' model of quantum double of $\cC$ when the number of simple classes are infinite.
As explained in~\citelist{\cite{MR1782145}\cite{DOI:10.1093/imrn/rnw304}}, the category of $*$-representations of the tube algebra $\cT(\cC)$ is equivalent to $\cZ(\indcat{\cC})$ as a C$^*$-category.

\smallskip
The structure of the annular algebra can be understood using  the \emph{monadic} formulation of the quantum double~\citelist{\cite{MR1966525}\cite{MR2869176}}.
Since there is a slight issue of encoding the relevant structures only using uniformly bounded families of morphisms, let us consider purely algebraic direct sums, so that we just present our construction on each direct summand.
Given $X \in \obj(\indcat{\cC})$, consider $\cZ(X) = \bigoplus_{S \in \cI}^\alg \bar{S} \otimes (X \otimes S)$ as a direct sum `without completion over the $S$'.
Then the monoidal structure of $\cC$ induces natural transformations $\mu\colon \cZ^2 \to \cZ$ and $\eta\colon \Id_\cC \to \cZ$ satisfying certain set of conditions analogous to monoids~\cite{MR1712872}.
To be specific, let us fix a standard solution $(R_S, \bar{R}_S)$ for each $S \in \cI$, and also an orthonormal basis $(w_i)_i$ of $\cC(U, S \otimes T)$ for each $S, T, U \in \cI$.
Then the part of $\mu$ from the summand $\bar{T} \otimes ((\bar{S} \otimes (X \otimes S)) \otimes T)$ to $\bar{U} \otimes (X \otimes U)$ is given by
\begin{equation}
\label{eq:monad-prod-prec-formula}
\sum_i (w_i^\vee \otimes (1_X \otimes w_i^*)) (1_{\bar T \otimes \bar S} \otimes \alpha_{X,S,T}) \alpha_{\bar T, \bar S, (X \otimes S) \otimes T}^{-1} (1_{\bar T} \otimes \alpha_{\bar{S}, X \otimes S, T}),
\end{equation}
where $w_i^\vee$ is calculated using standard solutions $(R_T, \bar{R}_T)$, $(R_U, \bar{R}_U)$, and \eqref{eq:std-sol-for-prod-with-assoc}.

An ind-object $X$ with a (non-unitary) half braiding $c$ admits a structure of $\cZ$-module, that is, a morphism $\cZ(X) \to X$ which is compatible with $\mu_X\colon \cZ^2(X) \to \cZ(X)$.
Concretely, the module structure is given by the collection of
$$
(1_X \otimes R_S^*) \alpha_{X \bar S, S} (c_{\bar S} \otimes 1_S) \alpha_{\bar S, X, S}^{-1} \colon \bar S \otimes (X \otimes S) \to X.
$$
The (non-unitary) Drinfeld center can be identified with the category $\cZ\modcat$ of $\cZ$-modules, which contain objects of the form $\cZ(X)$ as `free' $\cZ$-modules.
We have natural isomorphisms $\cA_{k, j} \cong \cC(X_{j}, \cZ(X_{k}))$, and the usual adjunction
$$
\cC(X_{j}, \cZ(X_{k})) \cong \Mor_{\cZ\modcat}(\cZ(X_{j}), \cZ(X_{k}))
$$
induces a $*$-algebra structure on $\cA$ from that of $\End_{\cZ\modcat}(\bigoplus_{j \in \cJ} \cZ(X_{j}))$, which is precisely the structure of tube algebra.

\section{Twisting and annular algebras}
\label{sec:twist-ann-alg}

Let $\cC$ be a $\Gamma$-graded rigid C$^*$-tensor category with simple unit.
For simplicity, we will assume that it is strict.
Given a normalized $3$-cocycle $\omega$ on $\Gamma$, our goal is to describe the tube (and more generally annular) algebra of $\cC^\omega$ in terms of that of $\cC$.
Let $\scX$ be a family of objects in $\cC$.
Since both $\cC$ and $\cC^\omega$ are identical as semisimple C$^*$-categories, $\cA = \cA(\scX)$ admit two different $*$-algebra structures.
Our goal is to extend the computation of~\cite{MR3699167} about $\cC_{\Gamma}^\omega$ to $\cC^\omega$ when $\scX$ is compatible with the grading as we specify below.
In order to avoid confusion we denote by $\bullet$ and $\#$ (resp.~by $\bbull$ and $\star$) the multiplication and the involution on $\cA$ induced by $\cC$ (resp.~by $\cC^\omega$).

\subsection{Fell bundle structure of the tube algebra}

As before, let $\cI$ be a set of representatives of isomorphism classes of simple objects.
Let us denote the map $\cI \to \Gamma$ corresponding to the $\Gamma$-grading by $S \mapsto \gamma_S$.

\begin{defn}
Let $\cC$ be a $\Gamma$-graded category.
A family of objects $\scX = \{ X_j \}_{j \in \cJ}$ of $\cC$ is \emph{adapted} to the grading when each $X_j$ is isomorphic to an object of $\cC_{\gamma_j}$ for some $\gamma_j \in \Gamma$.
\end{defn}

\begin{expl}
For example, $\cI$ is adapted for a trivial reason.
Another important example comes from the representation of a compact quantum group $G$.
Let $\{U_i\}_i$ be a set of representatives of irreducible representations, and $H_i$ be the underlying Hilbert space of $U_i$.
Then the family $\scX = \{ \bar{H}_i \otimes U_i \}_i$ in $\Rep G$ is adapted to the grading.
The associated annular algebra $\cA(\scX)$ is isomorphic to the Drinfeld double of $G$~\cite{arXiv:1511.06332}.
\end{expl}

We start with giving several convenient direct sum decompositions of the annular algebra $\cA = \cA(\scX)$ for an adapted family $\scX$.
Recall that we have direct sum decompositions
\begin{align*}
\cA &= \bigoplus_{j, k \in \cJ} \cA_{k, j},&
\cA_{k, j} &= \bigoplus_{S \in\cI} \cC (S \otimes X_j, X_k \otimes S)
\end{align*}
as vector spaces.
Note that the direct summand $\cA^S_{k, j} = \cC (S \otimes X_j, X_k \otimes S)$ is nonzero only if both $S \otimes X_j$ and $X_k \otimes S$ have subobjects isomorphic to a common element in $\cI$, which will then imply $\gamma_S \gamma_j = \gamma_k \gamma_S$.
Thus, if $\cA^S_{k, j}$ is nonzero then $\gamma_j$ and $\gamma_k$ are conjugate to each other implemented by $\gamma_S^{\pm 1}$.
This leads us to the following constructions.
For $\gamma, \eta, s \in \Gamma$ such that $s \gamma = \eta s$, set
\begin{align*}
\cA^s_{\eta, \gamma} &= \bigoplus_{\mathclap{\substack{j, k \in \cJ, S \in\cI \\ \gamma_j = \gamma, \gamma_k = \eta,\\ \gamma_S = s}}} \cA^S_{k, j}, &
\cA_{\eta, \gamma} &= \bigoplus_{\mathclap{t\colon t \gamma = \eta t}} \cA^t_{\eta, \gamma}.
\end{align*}
Finally, let $\Sigma$ be the set of conjugacy classes of $\Gamma$, and define $\cA_\sigma = \bigoplus_{\gamma,\eta \in \sigma} \cA_{\eta, \gamma}$ for $\sigma \in  \Sigma$.

\begin{lem}
\label{lem:direct-sum-over-conjug-classes}
The $*$-algebra $(\cA, \bullet, \#)$ decompose into the direct sum $\bigoplus_{\sigma \in \Sigma} \cA_\sigma$ of $*$-subalgebras.
Thus, every representation of the tube algebra is an orthogonal direct sum of representations of $\cA_\sigma$ for $\sigma \in \Sigma$.
\end{lem}

\begin{proof}
By the above discussion we have $\cA = \bigoplus_{\sigma \in \Sigma} \cA_\sigma$ as a vector space.
The multiplication formula \ref{annmult} tells us that $\cA_\sigma$ is closed under $\bullet$, and elements from different conjugacy classes are orthogonal.
More precisely $\cA_{t \eta t^{-1}, \eta}^t \cA_{s \gamma s^{-1}, \gamma}^s$ contains a nonzero element (if and) only if $\eta = s \gamma s^{-1}$, in which case it is in $\cA_{t s \gamma s^{-1} t^{-1}, \gamma}^{t s}$.
From the explicit formula, we can see that $\#$ sends elements of $\cA_{s \gamma s^{-1},\gamma}^s$ to $\cA_{\gamma, s \gamma s^{-1}}^{s^{-1}}$, so $\cA_\sigma$ is stable under involution.
\end{proof}

\begin{rem}
\label{rem:center-graded}
In terms of the Drinfeld center, the above direct sum decomposition corresponds to the fact that any object $(X, c)$ of $\cZ(\indcat{\cC})$ decomposes as a direct sum $\bigoplus_{\sigma \in \Sigma} (X_\sigma, c^\sigma)$ for some objects $X_\sigma \in \indcat{\bigoplus_{\gamma \in \sigma} \cC_\gamma}$ and unitary half braidings $c^\sigma$ on the $X_\sigma$.
Indeed, we can take a decomposition $X \cong \bigoplus_\gamma X_\gamma$ as an object in $\indcat{\cC}$ and put $X_\sigma = \bigoplus_{\gamma \in \sigma} X_\gamma$.
If $Y$ is an object in $\cC_\eta$, the morphism $c_Y$ from $\bigoplus_\sigma Y \otimes X_\sigma$ to $\bigoplus_\sigma X_\sigma \otimes Y$ has to be diagonally represented in $\sigma$, since there are no nontrivial morphisms between objects of $\indcat{\cC_{\eta \gamma}}$ and of $\indcat{\cC_{\gamma' \eta}}$ if $\gamma$ and $\gamma'$ are not conjugate in $\Gamma$.
Thus, each $X_\sigma$ inherits a half braiding $c^\sigma$ by restricting $c$.
\end{rem}

Let $\cG = \Gamma \ltimes_{\Ad} \Gamma$ be the action groupoid of $\Gamma$ acting on itself by the adjoint action.
The above argument shows that $\cA$ is a \emph{Fell bundle over $\cG$}~\citelist{\cite{MR1103378}\cite{MR1443836}}: to each $g = (s, \gamma)$ we have the direct summand $\cA_g = \cA_{s \gamma s^{-1}, \gamma}^s$ such that the $*$-algebra structure of $\cA = \bigoplus_{g \in \cG} \cA_g$ is compatible with the groupoid structure of $\cG$, in the sense that $\cA_{g} \cA_{g'} \subset \cA_{g g'}$ for $(g, g') \in \cG^{(2)}$ and $(\cA_g)^\# = \cA_{g^{-1}}$.

From now on, we will assume that the adapted family $ \mathscr X $ is full.
\begin{prop}
\label{prop:Ta-is-full-corner-in-TSigma}
Let $\sigma$ be a conjugacy class of $\Gamma$.
For any $a \in \sigma$, $\cA_{a, a}$ is strongly Morita equivalent to $\cA_\sigma$ via the completion of $\bigoplus_{b \in \sigma} \cA_{a, b}$ inside $\cA_\sigma$.
\end{prop}

\begin{proof}
Because of Proposition \ref{prop:ann-alg-indep-tensor-equiv} we may assume that $\cC$ is strict.
Let $b$ be an element of $\sigma$, and $j \in \cJ$ be any index satisfying $\gamma_j = b$.
The assertion holds if we can show that $\cA_{b, a} \cA_{a, b}$ contains the unit of $\cA_{j,j} $.
In fact, we claim that if $S$ is any irreducible object such that $a \gamma_S = \gamma_S b$, then the unit of $\cA_{j,j}$ is contained in the span of $\cA^{\bar{S}}_{j, k} \cA^S_{k, j}$ where $k$ runs through the indices of $\cJ$ satisfying $\gamma_k = \gamma_S \gamma_j \gamma_S^{-1}$.

Let $v\colon T \to S \otimes X_j \otimes \bar{S}$ be an isometry giving a irreducible subobject of $S \otimes X_j \otimes \bar{S}$, and put
\begin{align*}
f^{(v)} &= d(S)^{-\frac{1}{2}} (R_S^* \otimes 1_{j \otimes \bar{S}}) (1_{\bar{S}} \otimes v) \in \cC (\bar S \otimes T , X_j \otimes \bar S) ,\\ 
g^{(v)} &= d(S)^{-\frac{1}{2}} (v^* \otimes 1_{S}) (1_{S \otimes j} \otimes R_S) \in \cC (S \otimes X_j , T \otimes S),
\end{align*}
where $ R_S $ is a part of a standard solution to the conjugate equation of the dual pair $ (S , \bar S) $.
Then, by the product formula \ref{annmult}, we get
\[
\psi^{\bar S}_{j,T} (f^{(v)}) \bullet \psi^S_{T,j} (g^{(v)}) = d(S)^{-1} \psi^{\bar S \otimes S}_{j,j} \left( (R^*_S \otimes 1_{j \otimes \bar S \otimes S} ) (1_{\bar S } \otimes v v^* \otimes 1_S ) (1_{\bar S \otimes S \otimes j} \otimes R_S ) \right)
\]
Taking the summation on $v$ over a maximal system of orthogonal embeddings of irreducible subobjects of $S \otimes X_j \otimes \bar{S}$, we obtain
\[
\sum_v \psi^{\bar S}_{j,T} (f^{(v)}) \bullet \psi^S_{T,j} (g^{(v)}) = d(S)^{-1} \psi^{\bar S \otimes S}_{j,j} (R_S^*  \otimes 1_j \otimes R_S) = \psi^{\mathbbm 1}_{j,j} (1_j) .
\]
Since $ \mathscr X $ is full, each $T$, being irreducible, must be a subobject of some $ X_k \in \mathscr X $.
Again, $ \mathscr X $ is adapted and $ v $ is nonzero; so, indeed $ \gamma_k = \gamma_S \gamma_j \gamma_S^{-1}$.
\end{proof}

\begin{rem}
The above proposition shows that this bundle satisfies the assumption of~\cite{MR2446021}, namely,  for each $g = (s, \gamma)$ in $\cG$, $\cA_{g}$ is an imprimitivity bimodule between $\cA_{\dom(g)} = \cA_{\gamma, \gamma}^e$ and $\cA_{\codom(g)} = \cA_{s \gamma s^{-1}, s \gamma s^{-1}}^e$.
See \cite{MR3749336} for further implications on the structure of the C$^*$-envelope of $\cA$.
\end{rem}

\subsection{Tube algebra of twisted categories}

Continuing to denote $\cA = \cA(\scX)$ for a full and adapted family $\scX = \{X_j\}_{j \in \cJ}$, consider $\psi^S_{k,j'} (f_1) \in \cA^S_{k, j'}$ and $\psi^T_{j,i} (f_2) \in \cA^T_{j, i}$ for some indices $i, j, j',k \in \cJ$ and irreducible objects $S, T \in\cI$.
From the comparison of \eqref{annmult} for $\cC$ and $\cC^\omega$, we have
\[
\psi^S_{k, j'}(f_1) \bbull \psi^T_{j, i} (f_2) = \delta_{j,j'} \omega(\gamma_k, \gamma_S, \gamma_T) \bar{\omega}(\gamma_S, \gamma_j, \gamma_T) \omega(\gamma_S, \gamma_T, \gamma_i) \psi^S_{k, j}(f_1) \bullet \psi^T_{j, i} (f_2).
\]
Note that this equation is meaningful only when $\gamma_S$ (resp.~$\gamma_T$) conjugates $\gamma_{j^\prime}$ to $\gamma_k$ (resp.~$\gamma_i$ to $\gamma_j$); otherwise, both sides will become zero by the discussion preceding Lemma~\ref{lem:direct-sum-over-conjug-classes}.
In terms of the $\Gamma$-grading, we may rewrite the above equation as
\begin{equation}\label{eq:twistmult}
\chi_1 \bbull \chi_2 = \delta_{\gamma_2, g'_2} \omega(\gamma_3, s, t) \bar{\omega}(s, \gamma_2, t) \omega (s, t, \gamma_1) \chi_1  \bullet \chi_2
\end{equation}
for all $\chi_1 \in \cA^s_{\gamma_3, g'_2}$ and $\chi_2 \in \cA^t_{\gamma_2, \gamma_1}$.

Similarly, combining \eqref{eq:invol-on-ann-alg} for the category $\cC^\omega$ with Proposition \ref{prop:tw-cat-rigid}, we see that the $*$-structures are related by
\begin{equation}
\label{eq:twist-invol}
\chi^\star = \bar{\omega}(s^{-1}, s, \gamma_1) \bar{\omega}(s^{-1}, \gamma_2 s, s^{-1}) \bar{\omega}(\gamma_2, s, s^{-1}) \omega(s^{-1}, s, s^{-1}) \chi^\#
\end{equation}
for $\chi \in \cA_{\gamma_2, \gamma_1}^s$.

\begin{prop}[cf.~\cite{MR2443249}*{Theorem 3}]
Let us regard $\Gamma$ as a $\Gamma$-set by the adjoint action.
The normalized equivariant cochain $\Psi \in \rnC_\Gamma^2(\Gamma; \T)$ defined by
$$
\Psi[s, t](\gamma) = \omega(\gamma, s, t) \bar{\omega}(s, s^{-1} \gamma s, t) \omega(s, t, t^{-1} s^{-1} \gamma s t),
$$
is a cocycle.
\end{prop}

\begin{proof} By adapting \eqref{group-cocycle-cbdry} in the context of the adjoint action of $\Gamma$  on $X = \Gamma$, we see that a normalized equivariant $2$-cocycle in $\rnC_\Gamma^2 (\Gamma; \T)$ is a map $\Xi\colon \Gamma \times \Gamma \rightarrow \Map( \Gamma, \T )$ such that  
$$
\Xi[t, u](s^{-1} \gamma s) \ol{\Xi[s t, u](\gamma)} \Xi[s, t u](\gamma) \ol{\Xi[s, t](\gamma)} = 1.
$$
Expanding the left hand side for $\Xi = \Psi$, we obtain
\begin{equation*}
\begin{split}
&\uwave{\omega(s^{-1} \gamma s, t, u)} \uuline{\bar{\omega}(t, t^{-1} s^{-1} \gamma s t, u)} \dashuline{\omega(t, u, u^{-1} t^{-1} s^{-1} \gamma s t u)}\\
&\times \dotuline{\bar{\omega}(\gamma, s t, u)} \uuline{\omega(s t, (s t)^{-1} \gamma s t, u)} \dashuline{\bar{\omega}(s t, u, u^{-1} (s t)^{-1} \gamma s t u)} \\
&\times \dotuline{\omega(\gamma, s, t u)} \uwave{\bar{\omega}(s, s^{-1} \gamma s, t u)} \dashuline{\omega(s, t u, (t u)^{-1} s^{-1} \gamma s t u)}\\
&\times \dotuline{\bar{\omega}(\gamma, s, t)} \uwave{\omega(s, s^{-1} \gamma s, t)} \uuline{\bar{\omega}(s, t, t^{-1} s^{-1} \gamma s t)}.
\end{split}
\end{equation*}
By the $3$-cocycle relation, the terms with:
\begin{itemize}
\item dotted underline give $\omega(s, t, u) \bar{\omega}(\gamma s, t, u)$,
\item dashed underline give $\omega(s, t, t^{-1} s^{-1} \gamma s t u) \bar{\omega}(s, t, u)$,
\item wave underline give $\omega(\gamma s, t, u) \bar{\omega}(s, s^{-1} \gamma s t, u)$, and those with
\item double underline give $\omega(s, s^{-1} \gamma s t, u) \bar{\omega} ( s, t, t^{-1} s^{-1} \gamma s t u  )$.
\end{itemize}
Multiplying these terms we indeed obtain $1$.
\end{proof}

The corresponding groupoid $2$-cocycle on $\cG = \Gamma \ltimes_{\Ad} \Gamma$ (via Equation \ref{group-groupoid-link}) is given by 
\begin{equation}
\label{eq:groupoid-2-cocycle-prec-formula}
\psi(g_1, g_2) = \Psi[s_1, s_2](s_1 g_1 s_1^{-1}) = \omega(s_1 g_1 s_1^{-1}, s_1, s_2) \bar\omega(s_1, g_1, s_2) \omega(s_1, s_2, g_2)
\end{equation}
for $(g_1, g_2) \in \cG^{(2)}$ with $g_i = (s_i, g_i)$.
Thus, we can twist any Fell bundle $A$ over $\cG$ by $\psi$ to another Fell bundle $A_\psi = \langle a^{(\psi)} \mid a \in A_\gamma, \gamma \in \cG \rangle$ by
$$
a^{(\psi)} b^{(\psi)} = \psi(g_1, g_2) (a b)^{(\psi)} \quad ((g_1, g_2) \in \cG^{(2)}, a \in A_{g_1}, b \in A_{g_2}).
$$
Note however that $\psi$ is normalized only in the sense that $\psi(g_1, g_2) = 1$ when either $g_1$ or $g_2$ is in $\cG^{(0)}$.
Using this, one can easily check that $a^{(\psi)*} = \bar{\psi}(g^{-1}, g) a^{*(\psi)}$ for $a \in A_{g}$ gives a $*$-structure.
\begin{thm}
The $*$-algebra $(\cA, \bbull, \star)$ can be regarded as the $2$-cocycle twist of the Fell bundle $(\cA, \bullet, \#)$ by the $2$-cocycle $\psi$ on $\cG$.
\end{thm}

\begin{proof}
Take $g_i = (s_i, \gamma_i)$ ($i=1, 2$) in $\cG$ such that $(g_1, g_2) \in \cG^{(2)}$.
The product map $\cA_{g_1} \times \cA_{g_2} \to \cA_{g_1 g_2}$ in $(\cA, \bullet)_\psi$ is different from $\bullet$ by the factor of \eqref{eq:groupoid-2-cocycle-prec-formula}.
Since $\cA_{g_1} = \cA_{s_1 \gamma_1 s_1^{-1}, \gamma_1}^{s_1}$ and $\cA_{g_2} = \cA_{\gamma_1, \gamma_2}^{s_2}$, this is indeed the factor in \eqref{eq:twistmult}.

It remains to compare the involutions.
If $g = (s, \gamma_1) \in \cG$, the involution of $(\cA, \bullet, \#)_\psi$ on the summand $\cA_g$ is different from $\#$ by the factor of
$$
\bar{\psi}(g^{-1}, g) = \bar{\Psi}[s^{-1}, s](\gamma_1) = \bar{\omega}(\gamma_1, s^{-1}, s) \omega(s^{-1}, s \gamma_1 s^{-1}, s) \bar{\omega}(s^{-1}, s, \gamma_1).
$$
On the other hand, $\star$ on $\cA_g$ is different from $\#$ by the factor of
$$
\bar{\omega}(s^{-1}, s, \gamma_1) \bar{\omega}(s^{-1}, s \gamma_1, s^{-1}) \bar{\omega}(s \gamma_1 s^{-1}, s, s^{-1}) \omega(s^{-1}, s, s^{-1}).
$$
By the $3$-cocycle condition of $\bar\omega$ on quadruple $(s^{-1}, s \gamma_1, s^{-1}, s)$, we have
\begin{multline*}
\bar\omega(s \gamma_1, s^{-1}, s) \bar\omega(s^{-1}, s \gamma_1, s^{-1}) = \bar\omega(\gamma_1, s^{-1}, s)  \omega(s^{-1}, s \gamma_1 s^{-1}, s) \bar\omega(s^{-1}, s \gamma_1, e)\\
= \bar\omega(\gamma_1, s^{-1}, s) \omega(s^{-1}, s \gamma_1 s^{-1}, s).
\end{multline*}
Eliminating these terms and the common factors $\bar{\omega}(s^{-1}, s, \gamma_1)$ and $\bar{\omega}(s^{-1}, s \gamma_1, s^{-1})$, we are reduced to checking
$$
\bar{\omega}(s \gamma_1 s^{-1}, s, s^{-1}) \omega(s^{-1}, s, s^{-1}) = \bar{\omega}(s \gamma_1, s^{-1}, s).
$$
This equality follows from the $3$-cocycle condition (and normalization condition) of $\bar\omega$ on the quadruple $(s \gamma_1, s^{-1}, s, s^{-1})$.
\end{proof}

Let us also present a normalized version.

\begin{lem}
When $\psi$ is a $2$-cocycle on $\cG$ satisfying $\psi(g, \dom(g)) = 1 = \psi(\codom(g), g)$, we have $\psi(g, g^{-1}) = \psi(g^{-1}, g)$ for any $g \in \cG$.
\end{lem}

\begin{proof}
Using the $2$-cocycle condition on $(g, g^{-1}, g) \in \cG^{(3)}$, we obtain
$$
\psi(g^{-1}, g) \bar{\psi}(\codom(g), g) \psi(g, \dom(g)) \bar{\psi}(g, g^{-1}) = 1.
$$
Since $\psi$ is normalized, the middle two terms are trivial.
\end{proof}

It follows that we can choose a function $\xi \colon \cG \to \bT$ satisfying
$$
\psi(g, g^{-1}) = \xi(g)^2, \quad \xi(g) = \xi(g^{-1}), \quad \xi(x) = 1 \quad (x \in \cG^{(0)}).
$$
Then the cohomologous cocycle $\psi' = \psi . \delta^1 \bar{\xi}$ satisfies
\begin{gather*}
\psi'(g, g^{-1}) = \psi(g, g^{-1}) \bar{\xi}(g^{-1}) \bar{\xi}(g) \xi(\codom(g)) = 1, \\
\psi'(g, \dom(g)) = \psi(g, \dom(g)) = 1 = \psi'(\codom(g), g)
\end{gather*}
for any $g \in \cG$.
Thus, $\psi'$ is a normalized $2$-cocycle in the stronger sense, and $(\cA, \bbull, \star)$ is isomorphic to the twisting of $(\cA, \bullet, \#)$ by $\psi'$ by the map $\chi \mapsto \xi(g) \chi$ on $\cA_g$.

\medskip
Up to strong Morita equivalence, we can pick up one element from each orbit of $\cG$ (conjugacy class of $\Gamma$) and look at the stabilizers.
In our setting, this means that choosing $a \in \sigma$ for each $\sigma \in \Sigma$, and considering the centralizer subgroup $C_\Gamma(a) = \{\gamma \in \Gamma \mid a \gamma = \gamma a\}$ in  $\Gamma$.
Then $\psi$ induces a $2$-cocycle on $C_\Gamma(a)$,
$$
\vphi_a (s, t) =  \omega (a, s, t) \bar{\omega}(s, a, t) \omega (s, t, a),
$$
see~\citelist{\cite{MR1426907}\cite{MR3699167}*{Lemma~2.1}}.

The Fell bundle structure on $\cA$ implies that the $*$-algebras $(\cA_{a,a}, \bullet, \#)$ and $(\cA_{a,a}, \bbull, \star)$ are graded over $C_\Gamma(a)$.

\begin{cor}
\label{cor:cocycle-twist-over-centralizer}
Let $a$ be an element of $\Gamma$.
Then $(\cA_{a,a}, \bbull, \star)$ is isomorphic to the twist of the $C_\Gamma(a)$-graded $*$-algebra $(\cA_{a,a}, \bullet, \#)$ by the normalized $2$-cocycle $\vphi_a' = \vphi_a . d \bar{\xi}$.
\end{cor}

\begin{rem}
Since the $3$-cocycle $\omega$ of $\Gamma$ is normalized, the $2$-cocycle $\vphi_e$ on $C_\Gamma(e) = \Gamma$ turns out to be constant function $1$.
Consequently, if we consider the tube algebras, the two $*$-algebra structures are identical on $\cT_{e,e} = \cT_{\{e\}}$.
Note that $\cT_{e,e}$ contains the fusion algebra $\cT_{\mathbbm 1, \mathbbm 1}$ as a full corner by Proposition~\ref{prop:Ta-is-full-corner-in-TSigma}.
Recall that analytic properties $\cC$, such as, amenability, Haagerup and property (T), are defined using admissible representations of the fusion algebra, that is, those representations which extend to that of the whole tube algebra (see~\citelist{\cite{MR3406647}\cite{MR3447719}}).
Since $\cT_{e,e}$ is a $*$-ideal in $\cT$, any admissible representation of $\cT_{\mathbbm 1, \mathbbm 1}$ can extend up to $\cT_{e,e}$.
So, $\cC$ is amenable, has the Haagerup property, or has property (T) if and only if $\cC^\omega$ exhibits the corresponding properties.
\end{rem}

\begin{rem}
Let us describe the corresponding deformation of monad as in Section \ref{sec:ind-half-br-monad}.
Let $\cZ^\omega$ denote the corresponding monad of the tensor category $\cC^\omega$.
Suppose that $X$ has degree $\gamma$, and $S, T, U \in \cI$ respectively have degree {}$s, t, u$.
Then, from Proposition \ref{prop:tw-cat-rigid} and formulas \eqref{eq:monad-prod-prec-formula}, \eqref{eq:std-sol-for-prod-with-assoc}, we see that the morphism $\mu^\omega_X \colon (\cZ^\omega)^2 (X) \to \cZ^\omega(X)$ is given by the collection of morphisms from $\bar{T} \otimes ((\bar{S} \otimes (X \otimes S)) \otimes T)$ to $\bar{U} \otimes (X \otimes U)$ given by
$$
\omega(\gamma, s, t) \bar \omega(t^{-1}, s^{-1}, \gamma s t) \omega(s^{-1}, \gamma s, t) \bar \omega(t^{-1} s^{-1}, s, t) \omega(t^{-1}, s^{-1}, s) \mu_X
$$
which appears only when $U$ is isomorphic to a subobject of $S \otimes T$, hence in particular $u = s t$.
This factor can be interpreted as a $2$-cochain $\tilde \psi(g_1, g_2)$ on $\cG$ by setting $g_1 = (s, s^{-1} \gamma s)$ and $g_2 = (t, (s t)^{-1} \gamma s t)$.
As expected, it is cohomologous to $\psi$ via the coboundary of $\xi(g) = \omega(s^{-1}, s, \gamma)$ for $g = (s, \gamma)$, hence is a $2$-cocycle in particular.
\end{rem}

\section{Examples}\label{sec:examples}

\subsection{Grading by cyclic groups}

Let $\Gamma$ be a cyclic group.
Since $\rH^3(\Z; \T)$ is trivial, we concentrate on the case of finite cyclic group $\Gamma = \Z / n \Z$.
The consideration in the previous section become rather simple.
As $\Gamma$ is commutative, the direct sum decomposition of Lemma \ref{lem:direct-sum-over-conjug-classes} becomes $\cT(\cC) = \bigoplus_{a=0}^{n-1} \cT_{\{a\}}(\cC)$.
Moreover, as $\rH^2(\Gamma; \T)$ is trivial, Corollary~\ref{cor:cocycle-twist-over-centralizer} implies that $\cT_{\{a\}}(\cC^\omega)$ is isomorphic to $\cT_{\{a\}}(\cC)$.

\begin{expl}
The group $\Z/2\Z$ has essentially just one nontrivial $3$-cocycle, $\omega(a, b, c) = (-1)^{a b c}$.
For the categories of Example~\ref{expl:TL-and-SUq2} we have (by \cite{MR3340190}) $(\cC^\TL_{\delta,+1})^\omega = \cC^\TL_{\delta,-1} \cong \Rep \SU_q(2)$ for positive $q$ satisfying $q + q^{-1} = \delta$.
This explains why the computation of the spectrum of annular algebra for the category $\cC^\TL_{\delta,+1}$ in \cite{MR3447719} gives exactly the same answer as the computation for the Drinfeld double of $\SU_q(2)$ with positive $q$ in~\cite{MR1213303}.
\end{expl}

By the discussion made above,  there is a `natural' equivalence of the Drinfeld centers $\cZ(\indcat{\cC})$ and $\cZ(\indcat{\cC^\omega})$ as C$^*$-categories, compatible with the associated gradings by $\Z/n\Z$.
While a general description of the monoidal structure in terms of the tube algebra (cf.~\cite{MR3254423}) will give a description of $\cZ(\indcat{\cC^\omega})$ as a monoidal category, there is a more concrete description available in this setting, as follows.

It is well known that the $3$-cocycles on $\Z/n\Z$ can be represented (up to coboundary) by
\begin{equation}
\label{eq:3-cocycles-on-cyclic-group}
\omega^k(a, b, c) = \exp\left( 2 \pi \sqrt{-1} k \left(\left\lfloor \frac{a+b}{n} \right\rfloor - \left\lfloor \frac{a}{n} \right\rfloor - \left\lfloor \frac{b}{n} \right\rfloor\right) \frac{c}{n} \right) \quad (k = 0, 1, \ldots, n-1).
\end{equation}
We will work with this concrete form.
Thus, the associated $2$-cocycle is
$$
\vphi_a^k(s, t) = \exp\left( 2 \pi \sqrt{-1} k \left(\left\lfloor \frac{s+t}{n} \right\rfloor - \left\lfloor \frac{s}{n} \right\rfloor - \left\lfloor \frac{t}{n} \right\rfloor\right ) \frac{a}{n} \right).
$$
Let us introduce a new monoidal structure $\otimes^{(k)}$ on $\cZ(\indcat{\cC})$.
When $(X, c)$ and $(X', c')$ are objects of $\cZ(\indcat{\cC})$ such that $X \in \obj(\indcat{\cC}_a)$ and $X' \in \obj(\indcat{\cC}_b)$, we set
$$
(c \otimes^{(k)} c')_Y = \bar\omega^{2k}(a, b, s) (c \otimes c')_Y
$$
when $Y \in \obj(\cC_s)$ and extend by linearity to the general case, and put $(X, c) \otimes^{(k)} (X', c') = (X \otimes X', c \otimes^{(k)} c')$.
This still defines a structure of monoidal category (with the same associator as $
\cZ(\indcat{\cC})$) because $\omega$ is a $2$-cocycle in the first two variables.

\begin{prop}
\label{prop:new-mon-prod-on-center}
Let $\cC$ be a C$^*$-tensor category graded by $\Z/n\Z$, and $\omega^k$ be the $3$-cocycle as above.
Then $\cZ(\indcat{\cC^{\omega^k}})$ is unitarily monoidally equivalent to $(\cZ(\indcat{\cC}), \otimes^{(k)})^{\omega^k}$.
\end{prop}

\begin{proof}
Let $X$ be an object of $\indcat{\cC}$, and $(c_Y \colon Y \otimes X \to X \otimes Y)_Y$ be a unitary half braiding on $X$ (with respect to the structure of $\cC$).
We first define a half braiding $\tilde{c}$ on $X$ as an object of $\cC^\omega$.
By Remark~\ref{rem:center-graded}, it is enough to consider $(X, c) \in \obj(\cZ(\indcat{\cC}))$ with homogeneous object $X \in \obj(\indcat{\cC_a})$.

For $m \in \Z$, let $r(m)$ denote the unique integer $0 \le r(m) < n$ such that $m \equiv r(m) \bmod n$; in other words, $ r(m) = m - n \lfloor \frac m n \rfloor$.
When $Y$ is an object in $\cC_s$, put
$$
\tilde{c}_Y = \exp\left(-2 \pi \sqrt{-1} k\frac{r(s) r(a)}{n^2} \right) c_Y \colon Y \otimes X \to X \otimes Y,
$$
and extend it to general $Y$ using direct sum decomposition into homogeneous components.
We claim that $(X, \tilde{c})$ is a half braiding in $\cC^{\omega^k}$.
In view of \eqref{eq:half-br-mult-formula}, this amounts to verifying
\begin{multline*}
\omega^k(a, s, t) \exp\left(-2 \pi \sqrt{-1} k \frac{r(s) r(a)}{n^2} \right) \bar\omega^k(s, a, t) \exp\left(-\frac{2 \pi \sqrt{-1}}{n^2} r(t) r(a)\right) \omega^k(s, t, a) \\
= \exp\left(-2 \pi \sqrt{-1} k \frac{r(s + t) r(a)}{n^2} \right).
\end{multline*}
This is indeed the case, since the terms involving $\omega$ give $\omega^k(s, t, a)$, and we also have
\begin{equation}
\label{eq:res-and-floor}
r(s) + r(t) - r(s + t) = n \left( \flr{\frac{s+t}{n}} - \flr{\frac{s}{n}} - \flr{\frac{t}{n}}\right).
\end{equation}
This way we obtain a C$^*$-functor $F\colon \cZ(\indcat{\cC}) \to \cZ(\indcat{\cC}^{\omega^k})$ given by $F(X, c) = (X, \tilde{c})$.
It is an equivalence of categories, since we can `untwist' $\cC^{\omega^k}$ by $\bar{\omega}^k$ to recover $\cC$, and perform the same construction to produce an inverse of $F$.

It remains to show that $F$ can be enriched to a C$^*$-tensor functor if we first replace the monoidal structure of $\cZ(\indcat{\cC})$ by $\otimes^{(k)}$, and then twist it by $\omega^k$.
We claim that the natural unitary transformation $F_2\colon (X, \tilde{c}) \otimes (X', \tilde{c}') \to (X \otimes X', \widetilde{c \otimes^{(k)} c'})$ is simply represented by $1_{X \otimes X'}$.
Now, $1_{X \otimes X'}$ is a morphism of half braiding if and only if the diagram
$$
\begin{tikzcd}
Y \otimes (X \otimes X') \ar[d,"1"'] \ar[r, "\tilde{c} \otimes \tilde{c}'"] &  (X \otimes X') \otimes Y \ar[d,"1"]\\
Y \otimes (X \otimes X') \ar[r, "\widetilde{c \otimes^{(k)} c'}"'] & (X \otimes X') \otimes Y
\end{tikzcd}
$$
commutes.
To check this, as before we can also assume $X \in \obj(\indcat{\cC_a})$, $X' \in \obj(\indcat{\cC_b})$, and $Y \in \obj(\cC_s)$.
Then, the top row picks up the factor of
$$
\exp\left(- 2 \pi \sqrt{-1} k \frac{r(s) (r(a) + r(b))}{n^2} \right) \exp\left( - 2 \pi \sqrt{-1} k  \left( \flr{\frac{a+b}{n}} - \flr{\frac{a}{n}} - \flr{\frac{b}{n}} \right) \frac{s}{n} \right),
$$
while the bottom row picks up
$$
\exp\left( - 4 \pi \sqrt{-1} k \left( \flr{\frac{a+b}{n}} - \flr{\frac{a}{n}} - \flr{\frac{b}{n}} \right) \frac{s}{n} \right) \exp\left(- 2 \pi \sqrt{-1} k \frac{r(s) (r(a + b))}{n^2}  \right).
$$
These are indeed equal by \eqref{eq:res-and-floor}.

It remains to check the compatibility of $(F, 1_\un, F_2)$ with the associators of $(\cZ(\indcat{\cC}), \otimes^{(k)})^{\omega^k}$ and of $\cZ(\indcat{\cC}^{\omega^k})$, but it is obvious from the definitions.
\end{proof}

The structure $\otimes^{(k)}$ agrees with the original one when $n = 2$, or more generally if $n$ divides $2k$ in the parametrization of \eqref{eq:3-cocycles-on-cyclic-group}.
Moreover, the above proposition has an obvious parallel for $\cZ(\cC^{\omega^k})$, and also for nonunitary variants.

Suppose that $\cC$ is (resp.~unitarily) braided, or equivalently, that there is a (resp.~C$^*$-)tensor functor $F \colon \cC \to \cZ(\cC)$ of the form $X \mapsto (X, c)$.
If $n$ divides $2 k$, the above result implies that $X \mapsto (X, \tilde{c})$ is a (resp.~C$^*$-)tensor functor from $\cC^{\omega^k}$ to $\cZ(\cC^{\omega^k})$, hence $\cC^{\omega^k}$ is again (resp.~unitarily) braided.
If $n$ does not divide $2 k$ the situation looks more complicated, but at least we can say that $\cC^{\omega^k}$ is not braided when $\cC$ is either $\cC_{\Z/n\Z}$ or $\Rep \SU_q(n)$ as discussed below, cf.~\cite{MR3340190}*{Remark~4.4}.

\begin{expl}
\label{expl:kazh-wenzl-braided}
Let us describe the case of $\Rep \SU_q(n)$ for the root of unity $q = e^{\frac{\pi \sqrt{-1}}{m}}$, with $m \ge n - 1$.
Then it makes sense as a rigid C$^*$-tensor category $\cC$ (see, e.g.,~\cite{MR936086}) which we assume to be strict.
Its irreducible classes are parametrized by the dominant integral weights $\lambda$ of $\liesl_n$ such that $(\lambda + \rho, \alpha_{\max}) < m$, where $\rho$ is the half sum of positive roots, $\alpha_{\max}$ is the highest root, and $(\lambda, \mu)$ is the invariant inner product on the weight space normalized so that each root $\alpha$ satisfies $(\alpha, \alpha) = 2$.
If we denote the fundamental weights by $\omega_1, \ldots, \omega_{n-1}$, the above condition on $\lambda$ is equivalent to $\lambda = \sum_i \nu_i \omega_i$ with $\nu_i \ge 0$ and $\sum_i \nu_i \le m - n$.
Then $\cC$ is generated by the object $X_1$ (the `defining representation') corresponding to $\lambda = \omega_1$, and $X_1^{\otimes n}$ contains $\un$ with multiplicity $1$.
It follows that the universal grading group of $\cC$ is $\Z/n\Z$ such that $X_1$ belongs to the homogeneous component $\cC_1$.
Moreover, the classes $U^i$ ($i = 1, 2, \ldots, n - 1$) corresponding to $\lambda = (m - n)\omega_i$ exhaust the nontrivial invertible classes of $\cC$, and the class of $U^1$ satisfies $(U^1)^{\otimes k} \cong U^k$, $(U^1)^{\otimes n} \cong \un$~\cite{MR1741269}.
Now, since $\cC$ is a \emph{modular tensor category}, its Drinfeld center $\cZ(\cC)$ is equivalent to the Deligne product $\cC \boxtimes \cC^{\beta\opos}$ as a braided category, where $\beta\opos$ denotes taking the opposite braiding~\cite{MR1966525}.
Indeed, the equivalence $F\colon \cC \boxtimes \cC^{\beta\opos} \to \cZ(\cC)$ is given by $F(X \boxtimes X') = (X \otimes X', c)$, where $c = (c_Y)_Y$ is the half braiding on $X \otimes X'$ given by $c_Y = (1_X \otimes \beta_{X', Y}^{-1})(\beta_{Y, X} \otimes 1_{X'})$, with $\beta_{X, Y}\colon X \otimes Y \to Y \otimes X$ denoting the braiding of $\cC$.
In this picture the half braidings on the distinguished object $X_1$ are represented by $(X_1 \otimes U^i) \boxtimes U^{n-i}$ and $U^{n-i} \boxtimes (X_1 \otimes U^i)$ with the convention $U^0 = U^n = \un$.
In the notation of Proposition~\ref{prop:new-mon-prod-on-center}, the $n$-th tensor power of $(X_1 \otimes U^i) \boxtimes U^{n-i}$ with respect to $\otimes^{(k)}$ is the product of $F(X_1^{\otimes n} \boxtimes \un)$ with the half braiding on $\un$ given by $\exp(-4\pi \sqrt{-1} k s / n) \iota_Y$ for $Y \in \obj(\cC_s)$ up to the natural identification of $\un \otimes Y$ and $Y \otimes \un$ with $Y$.
Thus, it contains the trivial half braiding if and only if $n$ divides $2 k$, and the same can be said for $U^{n-i} \boxtimes (X_1 \otimes U^i)$.
In view of Proposition~\ref{prop:new-mon-prod-on-center}, there is a (C$^*$-)tensor functor from $\cC^{\omega^k}$ to $\cZ(\cC^{\omega^k})$ which is identity on objects, or equivalently $\cC^{\omega^k}$ is (unitarily) braided, if and only if $n$ divides $2 k$.
\end{expl}

\subsection{Free product}

For $\vlon \in \{ +, - \}$, let $\cC_\vlon$ be rigid semisimple C$^*$-tensor categories with simple units $\un_\vlon$ respectively.
Let $\cI_\vlon$ be a set of representatives of isomorphism classes of simple objects in $\cC_\vlon$, and $\Gamma_\vlon$ be the universal grading group.
Consider the \emph{free product} category $\cC = \cC_+ \ast \cC_{-}$.
Then the  universal grading group $\Gamma$ of $\cC$ is  isomorphic to the group free product $\Gamma_+ \ast \Gamma_-$.
Indeed, if  $\mcal W$ is the set of words (including the empty one) whose letters alternately belong to $\cI_+ \setminus \{ \un_+ \}$ and $\cI_- \setminus \{ \un_- \}$, then  $\cC$ has a set of representatives of simple objects enumerated by $\mcal W$, say $\cI = \{X_w\}_{w\in \mcal W}$  such that:
\begin{enumerate}
\item $X_\emptyset = \un$,
\item $X_{w_1}  \otimes X_{w_2} \cong X_{(w_1, {w_2})}$ whenever the last letter of $w_1$ and the first of $w_2$ have opposite signs (that is, if one is in $\cI_+$, then other is in $\cI_-$),
\item if $w_1 = (w'_2, Y),  w_2 = (Z, w'_2) \in \mcal W$ such that $Y$ and $Z$ have the same sign $\vlon = \pm$, then the direct sum decomposition of $X_{w_1} \otimes X_{w_2}$ into simple objects can be obtained inductively on the lengths of $w_1$ and $w_2$  by the following rule:
\begin{enumerate}
\item $X_{w_1} \otimes X_{w_2} \cong \displaystyle\bigoplus_{U\in \cI_\vlon} \left(X_{(w'_1, U, w'_2)}\right)^{\oplus \dim (\cC_\vlon (U, Y \otimes Z))}$ when $Y \ncong \bar{Z}$,
\item $X_{w_1} \otimes X_{w_2} \cong \left(X_{w'_1} \otimes X_{w'_2}\right) \oplus \biggl( \displaystyle\bigoplus_{U\in \cI_\vlon \setminus \{ \un\}} \left( X_{(w'_1, U, w'_2)} \right)^{\oplus \dim (\cC_\vlon (U, Y \otimes Z))} \biggr)$ when $Y \cong \bar{Z}$.
\end{enumerate}
\end{enumerate}
Using (i), (ii) and (iii), it is completely routine to check that the map
\[
\Gamma \ni [X_w] \mapsto Y_1 \cdots Y_n \in \Gamma_+ \ast \Gamma_-  \t { where } w =(Y_1, \ldots, Y_n) 
\]
is an isomorphism.

The cohomology of free product group is given by $\rH^3(\Gamma_+ * \Gamma_-; \bT) \cong \rH^3(\Gamma_+; \bT) \times \rH^3(\Gamma_-; \bT)$ (see, e.g., \cite{MR1269324}*{Section 6.2}).
This allows us to decompose the twisting procedure to the cases when $\omega$ on $\Gamma_+ * \Gamma_-$ is induced from a $3$-cocycle $\omega_0$ on $\Gamma_\vlon$ through the canonical quotient map $\Gamma_+ * \Gamma_- \to \Gamma_\vlon$.
This induces a groupoid homomorphism (a functor) $\cG_{\Gamma_+ * \Gamma_-} \to \cG_{\Gamma_\vlon}$, and the associated map $\rH^2(\cG_{\Gamma_\vlon}; \bT) \to \rH^2(\cG_{\Gamma_+ * \Gamma_-}; \bT)$.
If $\Psi_0$ is the $2$-cocycle on $\cG_{\Gamma_\vlon}$ associated with $\omega_0$, its pullback $\Psi$ on $\cG_{\Gamma_+ * \Gamma_-}$ is the one associated with $\omega$.

\subsection{Direct product}
\label{ssub:direct_product}

Next let us consider grading by direct product groups.
One source of such a structure is the Deligne product $\cC_1 \boxtimes \cC_2$ of $\cC_1$ and $\cC_2$.
In this case we have $\Ch(\cC_1 \boxtimes \cC_2) = \Ch(\cC_1) \times \Ch(\cC_2)$.

Suppose moreover that the grading is by a finite commutative group.
The $\bT$-valued $3$-cocycles on such groups are concretely characterized in \cite{MR3340190}*{Proposition A.3}, as follows.
Write $\Gamma = (\Z/n_1\Z) \times \cdots \times (\Z/n_k\Z)$ for $n_i \in \N$.
Then a set of generators on $\rH^3(\Gamma; \bT)$ is given by:
\begin{enumerate}
\item those of the form \eqref{eq:3-cocycles-on-cyclic-group} (with $k=1$) for some factor $\Z/n_i \Z$,
\item for distinct indices $i, j$:
$$
\phi_{ij}(a, b, c) = \exp\biggl(2 \pi \sqrt{-1} \left(\left\lfloor \frac{a_i+b_i}{n_i} \right\rfloor - \left\lfloor \frac{a_i}{n_i} \right\rfloor - \left\lfloor \frac{b_i}{n_i} \right\rfloor\right) \frac{c_j}{n_j} \biggr),
$$
\item for distinct indices $i, j, k$:
$$
\phi_{ijk}(a, b, c) = \exp\biggl( \frac{2 \pi \sqrt{-1}}{\gcd(n_i, n_j, n_k)} a_i b_j c_k \biggr).
$$
\end{enumerate}

\begin{expl}
Consider the case $\Gamma = (\Z/2\Z) \times (\Z/2\Z) \times (\Z/2\Z)$, and let $a = (0, 0, 1)$.
Then, the $3$-cocycle $\phi_{123}$ gives $\varphi_a(s, t) = \exp( \pi \sqrt{-1} s_1 t_2)$, or in a normalized form
$$
\varphi'_a(s, t) = \exp\left(\frac{\pi \sqrt{-1}}{2} (s_1 t_2 - s_2 t_1)\right).
$$
The tube algebra of $\cC_\Gamma$ is a commutative algebra of dimension $2^6$, while that of the twisted category $\cC_\Gamma^{\phi_{123}}$ has a noncommutative direct summand $\cT_{a,a}$.
In particular, the number of irreducible classes in $\cZ(\cC_\Gamma^{\phi_{123}})$ is smaller than that of $\cZ(\cC_\Gamma)$.
\end{expl}


\begin{bibdiv}
\begin{biblist}

\bib{MR2146291}{article}{
      author={Altschuler, D.},
      author={Coste, A.},
      author={Maillard, J.-M.},
       title={Representation theory of twisted group double},
        date={2004},
        ISSN={0182-4295},
     journal={Ann. Fond. Louis de Broglie},
      volume={29},
      number={4},
       pages={681\ndash 694},
      eprint={\href{http://arxiv.org/abs/hep-th/0309257}{{\tt
  arXiv:hep-th/0309257 [hep-th]}}},
      review={\MR{2146291}},
}

\bib{MR1070067}{article}{
      author={B\'antay, P.},
       title={Orbifolds and {H}opf algebras},
        date={1990},
        ISSN={0370-2693},
     journal={Phys. Lett. B},
      volume={245},
      number={3-4},
       pages={477\ndash 479},
         url={https://doi.org/10.1016/0370-2693(90)90676-W},
      review={\MR{1070067}},
}

\bib{MR3699167}{article}{
      author={Bisch, Dietmar},
      author={Das, Paramita},
      author={Ghosh, Shamindra~Kumar},
      author={Rakshit, Narayan},
       title={Tube algebra of group-type subfactors},
        date={2017},
        ISSN={0129-167X},
     journal={Internat. J. Math.},
      volume={28},
      number={10},
       pages={1750069, 32},
      eprint={\href{http://arxiv.org/abs/1701.00097}{\texttt{arXiv:1701.00097
  [math.OA]}}},
         url={https://doi.org/10.1142/S0129167X17500690},
         doi={10.1142/S0129167X17500690},
      review={\MR{3699167}},
}

\bib{MR1741269}{article}{
      author={Brugui{\`e}res, Alain},
       title={Cat\'egories pr\'emodulaires, modularisations et invariants des
  vari\'et\'es de dimension 3},
        date={2000},
        ISSN={0025-5831},
     journal={Math. Ann.},
      volume={316},
      number={2},
       pages={215\ndash 236},
         url={http://dx.doi.org/10.1007/s002080050011},
         doi={10.1007/s002080050011},
      review={\MR{1741269 (2001d:18009)}},
}

\bib{MR2869176}{article}{
      author={Brugui{{\`e}}res, Alain},
      author={Virelizier, Alexis},
       title={Quantum double of {H}opf monads and categorical centers},
        date={2012},
        ISSN={0002-9947},
     journal={Trans. Amer. Math. Soc.},
      volume={364},
      number={3},
       pages={1225\ndash 1279},
         url={http://dx.doi.org/10.1090/S0002-9947-2011-05342-0},
         doi={10.1090/S0002-9947-2011-05342-0},
      review={\MR{2869176}},
}

\bib{MR3254423}{article}{
      author={Das, Paramita},
      author={Ghosh, Shamindra~Kumar},
      author={Gupta, Ved~Prakash},
       title={Drinfeld center of planar algebra},
        date={2014},
        ISSN={0129-167X},
     journal={Internat. J. Math.},
      volume={25},
      number={8},
       pages={1450076 (43 pages)},
         url={http://dx.doi.org/10.1142/S0129167X14500761},
         doi={10.1142/S0129167X14500761},
      review={\MR{3254423}},
}

\bib{MR3447719}{article}{
      author={Ghosh, Shamindra~Kumar},
      author={Jones, Corey},
       title={Annular representation theory for rigid {$C^*$}-tensor
  categories},
        date={2016},
        ISSN={0022-1236},
     journal={J. Funct. Anal.},
      volume={270},
      number={4},
       pages={1537\ndash 1584},
      eprint={\href{http://arxiv.org/abs/1502.06543}{\texttt{arXiv:1502.06543
  [math.OA]}}},
         url={http://dx.doi.org/10.1016/j.jfa.2015.08.017},
         doi={10.1016/j.jfa.2015.08.017},
      review={\MR{3447719}},
}

\bib{MR3749336}{article}{
      author={Ionescu, Marius},
      author={Kumjian, Alex},
      author={Sims, Aidan},
      author={Williams, Dana~P.},
       title={A stabilization theorem for {F}ell bundles over groupoids},
        date={2018},
        ISSN={0308-2105},
     journal={Proc. Roy. Soc. Edinburgh Sect. A},
      volume={148},
      number={1},
       pages={79\ndash 100},
      eprint={\href{http://arxiv.org/abs/1512.06046}{\texttt{MR3749336
  [math.OA]}}},
         url={https://doi.org/10.1017/S0308210517000129},
         doi={10.1017/S0308210517000129},
      review={\MR{3749336}},
}

\bib{MR1782145}{article}{
      author={Izumi, Masaki},
       title={The structure of sectors associated with {L}ongo-{R}ehren
  inclusions. {I}. {G}eneral theory},
        date={2000},
        ISSN={0010-3616},
     journal={Comm. Math. Phys.},
      volume={213},
      number={1},
       pages={127\ndash 179},
         url={http://dx.doi.org/10.1007/s002200000234},
         doi={10.1007/s002200000234},
      review={\MR{1782145 (2002k:46160)}},
}

\bib{MR1929335}{incollection}{
      author={Jones, Vaughan F.~R.},
       title={The annular structure of subfactors},
        date={2001},
   booktitle={Essays on geometry and related topics, {V}ol. 1, 2},
      series={Monogr. Enseign. Math.},
      volume={38},
   publisher={Enseignement Math., Geneva},
       pages={401\ndash 463},
  eprint={\href{http://arxiv.org/abs/math/0105071}{\texttt{arXiv:math/0105071
  [math.OA]}}},
      review={\MR{1929335 (2003j:46094)}},
}

\bib{MR1237835}{incollection}{
      author={Kazhdan, David},
      author={Wenzl, Hans},
       title={Reconstructing monoidal categories},
        date={1993},
   booktitle={I. {M}. {G}el\cprime fand {S}eminar},
      series={Adv. Soviet Math.},
      volume={16},
   publisher={Amer. Math. Soc.},
     address={Providence, RI},
       pages={111\ndash 136},
      review={\MR{1237835 (95e:18007)}},
}

\bib{MR1443836}{article}{
      author={Kumjian, Alex},
       title={Fell bundles over groupoids},
        date={1998},
        ISSN={0002-9939},
     journal={Proc. Amer. Math. Soc.},
      volume={126},
      number={4},
       pages={1115\ndash 1125},
      eprint={\href{http://arxiv.org/abs/math/9607230}{{\tt arXiv:math/9607230
  [math.OA]}}},
         url={http://dx.doi.org/10.1090/S0002-9939-98-04240-3},
         doi={10.1090/S0002-9939-98-04240-3},
      review={\MR{1443836 (98i:46055)}},
}

\bib{MR1332979}{article}{
      author={Longo, R.},
      author={Rehren, K.-H.},
       title={Nets of subfactors},
        date={1995},
        ISSN={0129-055X},
     journal={Rev. Math. Phys.},
      volume={7},
      number={4},
       pages={567\ndash 597},
  eprint={\href{http://arxiv.org/abs/hep-th/9411077}{\texttt{arXiv:hep-th/9411077
  [hep-th]}}},
         url={http://dx.doi.org/10.1142/S0129055X95000232},
         doi={10.1142/S0129055X95000232},
        note={Workshop on Algebraic Quantum Field Theory and Jones Theory
  (Berlin, 1994)},
      review={\MR{1332979 (96g:81151)}},
}

\bib{MR1712872}{book}{
      author={Mac~Lane, Saunders},
       title={Categories for the working mathematician},
     edition={Second},
      series={Graduate Texts in Mathematics},
   publisher={Springer-Verlag, New York},
        date={1998},
      volume={5},
        ISBN={0-387-98403-8},
      review={\MR{1712872 (2001j:18001)}},
}

\bib{MR1966525}{article}{
      author={M{\"u}ger, Michael},
       title={From subfactors to categories and topology. {II}. {T}he quantum
  double of tensor categories and subfactors},
        date={2003},
        ISSN={0022-4049},
     journal={J. Pure Appl. Algebra},
      volume={180},
      number={1-2},
       pages={159\ndash 219},
  eprint={\href{http://arxiv.org/abs/math/0111205}{\texttt{arXiv:math/0111205
  [math.CT]}}},
         url={http://dx.doi.org/10.1016/S0022-4049(02)00248-7},
         doi={10.1016/S0022-4049(02)00248-7},
      review={\MR{1966525 (2004f:18014)}},
}

\bib{MR2446021}{article}{
      author={Muhly, Paul~S.},
      author={Williams, Dana~P.},
       title={Equivalence and disintegration theorems for {F}ell bundles and
  their {$C^*$}-algebras},
        date={2008},
        ISSN={0012-3862},
     journal={Dissertationes Math. (Rozprawy Mat.)},
      volume={456},
       pages={1\ndash 57},
      eprint={\href{http://arxiv.org/abs/0806.1022}{{\tt arXiv:0806.1022
  [math.OA]}}},
         url={http://dx.doi.org/10.4064/dm456-0-1},
         doi={10.4064/dm456-0-1},
      review={\MR{2446021 (2010b:46146)}},
}

\bib{MR3204665}{book}{
      author={Neshveyev, Sergey},
      author={Tuset, Lars},
       title={Compact quantum groups and their representation categories},
      series={Cours Sp{\'e}cialis{\'e}s [Specialized Courses]},
   publisher={Soci{\'e}t{\'e} Math{\'e}matique de France, Paris},
        date={2013},
      volume={20},
        ISBN={978-2-85629-777-3},
      review={\MR{3204665}},
}

\bib{MR3340190}{article}{
      author={Neshveyev, Sergey},
      author={Yamashita, Makoto},
       title={Twisting the {$q$}-deformations of compact semisimple {L}ie
  groups},
        date={2015},
        ISSN={0025-5645},
     journal={J. Math. Soc. Japan},
      volume={67},
      number={2},
       pages={637\ndash 662},
      eprint={\href{http://arxiv.org/abs/1305.6949}{\texttt{arXiv:1305.6949
  [math.OA]}}},
         url={http://dx.doi.org/10.2969/jmsj/06720637},
         doi={10.2969/jmsj/06720637},
      review={\MR{3340190}},
       label={NY15a},
}

\bib{MR3509018}{article}{
      author={Neshveyev, Sergey},
      author={Yamashita, Makoto},
       title={Drinfeld {C}enter and {R}epresentation {T}heory for {M}onoidal
  {C}ategories},
        date={2016},
        ISSN={0010-3616},
     journal={Comm. Math. Phys.},
      volume={345},
      number={1},
       pages={385\ndash 434},
      eprint={\href{http://arxiv.org/abs/1501.07390}{\texttt{arXiv:1501.07390
  [math.OA]}}},
         url={http://dx.doi.org/10.1007/s00220-016-2642-7},
         doi={10.1007/s00220-016-2642-7},
      review={\MR{3509018}},
}

\bib{arXiv:1511.06332}{misc}{
      author={Neshveyev, Sergey},
      author={Yamashita, Makoto},
       title={A few remarks on the tube algebra of a monoidal category},
         how={preprint},
        date={2015},
      eprint={\href{http://arxiv.org/abs/1511.06332}{\texttt{arXiv:1511.06332
  [math.OA]}}},
        note={to appear in Proc. Edinb.  Math. Soc.},
       label={NY15b},
}

\bib{MR1317353}{incollection}{
      author={Ocneanu, Adrian},
       title={Chirality for operator algebras},
        date={1994},
   booktitle={Subfactors ({K}yuzeso, 1993)},
   publisher={World Sci. Publ., River Edge, NJ},
       pages={39\ndash 63},
      review={\MR{MR1317353 (96e:46082)}},
}

\bib{DOI:10.1093/imrn/rnw304}{article}{
      author={Popa, Sorin},
      author={Shlyakhtenko, Dimitri},
      author={Vaes, Stefaan},
       title={Cohomology and {$L^2$}-{B}etti numbers for subfactors and
  quasi-regular inclusions},
        date={2018},
     journal={Internat. Math. Res. Notices},
      volume={2018},
      number={8},
       pages={2241\ndash 2331},
      eprint={\href{http://arxiv.org/abs/1511.07329}{\texttt{arXiv:1511.07329
  [math.OA]}}},
         doi={10.1093/imrn/rnw304},
}

\bib{MR3406647}{article}{
      author={Popa, Sorin},
      author={Vaes, Stefaan},
       title={Representation theory for subfactors, {$\lambda$}-lattices and
  {$\rm C^*$}-tensor categories},
        date={2015},
        ISSN={0010-3616},
     journal={Comm. Math. Phys.},
      volume={340},
      number={3},
       pages={1239\ndash 1280},
      eprint={\href{http://arxiv.org/abs/1412.2732}{\texttt{arXiv:1412.2732
  [math.OA]}}},
         url={http://dx.doi.org/10.1007/s00220-015-2442-5},
         doi={10.1007/s00220-015-2442-5},
      review={\MR{3406647}},
}

\bib{MR1213303}{article}{
      author={Pusz, Wies{\l}aw},
       title={Irreducible unitary representations of quantum {L}orentz group},
        date={1993},
        ISSN={0010-3616},
     journal={Comm. Math. Phys.},
      volume={152},
      number={3},
       pages={591\ndash 626},
         url={https://projecteuclid.org/euclid.cmp/1104252519},
         doi={10.1007/BF02096620},
      review={\MR{1213303 (94c:81085)}},
}

\bib{MR2674592}{book}{
      author={Turaev, Vladimir},
       title={Homotopy quantum field theory},
      series={EMS Tracts in Mathematics},
   publisher={European Mathematical Society (EMS), Z{\"u}rich},
        date={2010},
      volume={10},
        ISBN={978-3-03719-086-9},
         url={http://dx.doi.org/10.4171/086},
         doi={10.4171/086},
        note={Appendix 5 by Michael M{\"u}ger and Appendices 6 and 7 by Alexis
  Virelizier},
      review={\MR{2674592}},
}

\bib{MR1269324}{book}{
      author={Weibel, Charles~A.},
       title={An introduction to homological algebra},
      series={Cambridge Studies in Advanced Mathematics},
   publisher={Cambridge University Press, Cambridge},
        date={1994},
      volume={38},
        ISBN={0-521-43500-5; 0-521-55987-1},
         url={https://doi.org/10.1017/CBO9781139644136},
      review={\MR{1269324}},
}

\bib{MR936086}{article}{
      author={Wenzl, Hans},
       title={Hecke algebras of type {$A_n$} and subfactors},
        date={1988},
        ISSN={0020-9910},
     journal={Invent. Math.},
      volume={92},
      number={2},
       pages={349\ndash 383},
         url={http://dx.doi.org/10.1007/BF01404457},
         doi={10.1007/BF01404457},
      review={\MR{936086 (90b:46118)}},
}

\bib{MR0255771}{article}{
      author={Westman, Joel~J.},
       title={Cohomology for ergodic groupoids},
        date={1969},
        ISSN={0002-9947},
     journal={Trans. Amer. Math. Soc.},
      volume={146},
       pages={465\ndash 471},
         url={https://doi.org/10.2307/1995186},
      review={\MR{0255771}},
}

\bib{MR2443249}{article}{
      author={Willerton, Simon},
       title={The twisted {D}rinfeld double of a finite group via gerbes and
  finite groupoids},
        date={2008},
        ISSN={1472-2747},
     journal={Algebr. Geom. Topol.},
      volume={8},
      number={3},
       pages={1419\ndash 1457},
      eprint={\href{http://arxiv.org/abs/math/0503266}{{\tt arXiv:math/0503266
  [math.QA]}}},
         url={https://doi.org/10.2140/agt.2008.8.1419},
      review={\MR{2443249}},
}

\bib{MR1426907}{article}{
      author={Witherspoon, S.~J.},
       title={The representation ring of the twisted quantum double of a finite
  group},
        date={1996},
        ISSN={0008-414X},
     journal={Canad. J. Math.},
      volume={48},
      number={6},
       pages={1324\ndash 1338},
         url={https://doi.org/10.4153/CJM-1996-070-6},
      review={\MR{1426907}},
}

\bib{MR890482}{article}{
      author={Woronowicz, S.~L.},
       title={Twisted {${\rm SU}(2)$} group. {A}n example of a noncommutative
  differential calculus},
        date={1987},
        ISSN={0034-5318},
     journal={Publ. Res. Inst. Math. Sci.},
      volume={23},
      number={1},
       pages={117\ndash 181},
         url={http://dx.doi.org/10.2977/prims/1195176848},
         doi={10.2977/prims/1195176848},
      review={\MR{890482 (88h:46130)}},
}

\bib{MR1103378}{incollection}{
      author={Yamagami, Shigeru},
       title={On primitive ideal spaces of {$C^*$}-algebras over certain
  locally compact groupoids},
        date={1990},
   booktitle={Mappings of operator algebras ({P}hiladelphia, {PA}, 1988)},
      series={Progr. Math.},
      volume={84},
   publisher={Birkh\"auser Boston},
     address={Boston, MA},
       pages={199\ndash 204},
      review={\MR{1103378 (92j:46110)}},
}

\end{biblist}
\end{bibdiv}
\end{document}